\documentclass[11pt]{amsart}
\usepackage{a4wide,amsfonts,amsmath,amssymb,amsthm,epsfig,cite,graphicx,hyperref,
color, esint,fancyhdr, enumerate, latexsym,amsrefs,physics}

\newtheorem{thm}{Theorem}[section]

\newtheorem{lem}[thm]{Lemma}
\newtheorem{cor}[thm]{Corollary}
\theoremstyle{definition}
\newtheorem{defn}[thm]{Definition}

\newcommand{\ep}{\epsilon}

\newcommand{\clb}{\mathbb{B}}

\newcommand{\clcn}{\mathbb{C}^{n}}

\renewcommand{\Re}{\operatorname{Re}}

\numberwithin{equation}{section}

\hypersetup{
    colorlinks,
    citecolor=blue,
    filecolor=black,
    linkcolor=red,
    urlcolor=black
}

\title{The Bergman-Fridman invariant on some classes of pseudoconvex domains}
\keywords{Bergman distance, Scaling, Fridman invariant}
\subjclass{Primary: 32F45; Secondary: 32A25, 32A36}
\author{Rahul Kumar and Prachi Mahajan}

\address{Rahul Kumar: Department of Mathematics, Indian Institute of Technology Bombay, Powai, Mumbai 400076, India}
\email{rahulkramen@gmail.com}

\address{Prachi Mahajan: Department of Mathematics, Indian Institute of Technology Bombay, Powai, Mumbai 400076, India}
\email{prachi.mjn@iitb.ac.in}

\begin{document}
\maketitle
\begin{abstract}
We study the boundary behaviour of a variant of the Fridman's function (defined in terms of the Bergman metric) on Levi corank one domains, strongly pseudoconvex domains, smoothly bounded convex domains in $ \mathbb{C}^n $ and polyhedral domains in $ \mathbb{C}^2 $. Two examples are given to show 
that this invariant detects (local) strong  pseudoconvexity of domain from its boundary behaviour.
\end{abstract}

\maketitle


\section{Introduction}

Denote by $ \mathbb{B}^n $ the unit ball of $ \mathbb{C}^n $. For a bounded domain $ X \subset \mathbb{C}^n $, and a point $ p \in X $, denote by $B^b_X (p, r) $, the ball around $ p $ of radius $ r >0 $ with respect to the Bergman distance on $ X $. Let $ \mathcal{R}(p) $ be the set of all radii $ r > 0 $ such that a holomorphic imbedding $ f : \mathbb{B}^n \rightarrow X $ with $ B^b_X (p, r) \subset f (\mathbb{B}^n) $ can be found. Then $ \mathcal{R}(p) $ is non-empty, and the function
\begin{equation*}
 h^b_X (p) := \inf_{r \in \mathcal{R}(p)} \frac{1}{r},
\end{equation*}
is well-defined. That this is a biholomorphic invariant follows from its deﬁnition and the fact that Bergman distance itself is invariant under biholomorphisms.

The above construction was introduced by B. L. Fridman in \cite{Frid1}, \cite{Frid2} -- on a given Kobayashi hyperbolic complex manifold of dimension $ n $, say $ X $, wherein the  balls $ B^k_X(p,r) $ with respect to the Kobayashi distance on $ X $ were considered in place of $ B^b_X(p,r) $. The resultant function $ h^k_X (\cdot) $ is a non-negative real-valued function on $ X $ which is invariant under biholomorphisms. Fridman’s definition of $ h^k_X $ is flexible enough to accommodate metrics other that the Kobayashi metric – this is already evident in \cite{MV} where the Carath\'{e}odory metric on $ X $ was considered; also see \cite{NV}. Fridman's construction induces in a similar way the invariant $ h^b_X $ on a bounded domain $ X \subset \mathbb{C}^n $ -- which will be referred to as the Bergman-Fridman invariant henceforth. We shall focus on the Bergman metric exclusively and write $ h_X (\cdot)$ and $ B_X(\cdot, r) $ instead of $ h^b_X (\cdot)$ and $ B^b_X(\cdot, r) $ for notational convenience.

The first purpose of this note is to investigate the boundary behaviour of the Bergman-Fridman invariant for a variety of pseudoconvex domains. The notation $ (z_1, z') $ is used for $ (z_1, z_2, \ldots, z_n) \in \mathbb{C}^n $ throughout.

\begin{thm} \label{T0}
 Let $ D $ be a bounded domain in $ \mathbb{C}^n $ and $ \{p^j\} $ be a sequence of points in $ D $ that converges to $ p^0 \in \partial D $. 
 \begin{enumerate}
 
  \item [(i)] If $ D $ is a $ C^{\infty}$-smooth Levi corank one domain and $ \Gamma $ is a nontangential cone in $ D $ with vertex at $ p^{0}$, then 
$ \displaystyle\lim_{\Gamma \ni p^j \rightarrow p^0} h_D(z) \rightarrow h_{D_{\infty}} \left( (-1,0') \right) $,
where 
\begin{equation*}
 D_{\infty}=\Big\{ z \in \mathbb{C}^n : 2 \Re z_1 + P_{2m} \left(z_n, \overline{z}_n \right) + \sum_{j=2}^{n-1} \vert z_j \vert^2 < 0 \Big\},
\end{equation*}
$m\geq 1$ is a positive integer and $ P_{2m} \left(z_n, \overline{z}_n \right) $ is a subharmonic polynomial of degree at most $ 2m $ without any harmonic terms, and $ 2m $ is the $1$-type of $ \partial D $ at $ p^0 $.

\item[(ii)] If $ D $ is a $ C^2 $-smooth strongly pseudoconvex domain,
 then $ \lim_{ p^j \rightarrow p^0} h_D(z) \rightarrow 0 $.

\item [(iii)] If $ D $ is a $ C^{\infty} $-smooth convex domain,
then $ \lim_{ p^j \rightarrow p^0} h_D(z) \rightarrow h_{D_{\infty}} \left( (-1,0') \right) $, where $ D_{\infty}$ is the limiting domain associated to the domain $D$ and the sequence $ p^j \rightarrow p^0 $. 

\item[(iv)] If $n=2$ and $ D \subset \mathbb{C}^2$ is a strongly pseudoconvex polyhedral domain, then $ \lim_{ p^j \rightarrow p^0} h_D(z) \rightarrow h_{D_{\infty}} \left( q^0 \right) $, where $ D_{\infty}$ is the limiting domain associated to the domain $D$ and $ q^0 \in D_{\infty} $.

 \end{enumerate}

\end{thm}

A smoothly bounded pseudoconvex finite type domain $ D \subset \mathbb{C}^n $ is said to Levi corank one if its Levi form
at every boundary point has at most one degenerate eigenvalue. The class of Levi corank one domains includes smooth pseudoconvex finite type domains in $ \mathbb{C}^2 $ and strongly pseudoconvex domains in $ \mathbb{C}^n $. Note that the boundary $ \partial D $ can \textit{a priori} be of inﬁnite type near the boundary point $ p^0 $ in Theorem \ref{T0} (iii). The precise definition of polyhedral domains dealt with in (iv) above is postponed to Section \ref{poly} - but note that such domains arise as the intersection of finitely many strongly pseudoconvex domains, where the singularities of the boundary are generated only by a normal crossing.  

It is worthwhile mentioning that in cases (ii)-(iv), $ p^j $ is any \textit{arbitrary} sequence converging to $ p^0 \in \partial D $ (and not restricted to a non-tangential approach region). In each of the four cases listed above, by scaling $ D $ along the sequence $ \{p^j\} $, we obtain a sequence of domains $ D^j $, each biholomorphic to $ D $, such that $ D^j $ converge in the local Hausdorff sense to a \textit{limit} domain $ D_{\infty} $. Moreover, by the scaling technique, the sequence $ \{ p^j\} $ is sent to a convergent sequence in $ D_{\infty} $. If $ D $ is Levi corank one or convex finite type domain, then the associated limit domain $ D_{\infty} $ is a polynomial domain. Understanding the boundary limits for $ h_D $ amounts to controlling the Bergman kernels $ K_{D^j} $ and the Bergman distance $ d^b_{D^j} $ for the scaled domains $ D^j $. So, the main issue here is to show that the Bergman distance is stable under the scaling process (refer Theorem \ref{L5}). 
Further, it turns out that the non-vanishing of the Bergman kernel $ K_{D_{\infty}} $ along the diagonal and  Bergman completeness of $ D_{\infty} $ in each of the cases considered in Theorem \ref{T0} are vital for studying the possible convergence of $ d^b_{D^j} $.

It is natural to ask if the converse implication to Theorem \ref{T0} part (ii) holds, i.e., given a bounded domain $ D $ with smooth boundary, does $ \lim_{z \rightarrow \partial D} h_D(z) = 0 $ imply that
$ D $ is strongly pseudoconvex? An affirmative answer is 
given here when $ D \subset \mathbb{C}^n $ is either a Levi corank one domain or a convex domain. In particular, the Bergman-Fridman invariant function has a good efficacy in characterizing strong pseudoconvexity for a given domain if its boundary behaviour is known. More precisely,

\begin{thm} \label{T4}
Let $D \subset \mathbb C^n$ be a bounded Levi corank one domain with $ C^{\infty}$-smooth boundary. Let $ p^0 \in \partial D $ be such that $h_D(z) \rightarrow 0$ as $ z \rightarrow p^0 $, then $\partial D$ is strongly pseudoconvex near $p^0$. 
\end{thm}

\begin{thm} \label{T5}
Let $D \subset \mathbb C^n$ be a bounded convex domain with $C^{\infty}$-smooth boundary. Then $\partial D$ is strongly pseudoconvex if $h_D(z) \rightarrow 0$ as $z \rightarrow \partial D$.
\end{thm}

Just as the Bergman metric admits a localisation near local holomorphic peak points of a bounded pseudoconvex domain, the same holds even for $ h_D $.

\begin{thm}\label{n10}
Let $D \subset \mathbb{C}^n $ be a bounded pseudoconvex domain and $z^{0} \in \partial{D}$ be a local holomorphic peak point of $ D $.  For a sufficiently small neighbourhood $U$ of $z^{0}$, 
\begin{enumerate}
 \item [(i)] $ \lim_{z \rightarrow z^0} h_D (z) = 0 $ if and only if $ \lim_{z \rightarrow z^0} h_{U \cap D} (z) = 0 $.
 \item [(ii)] If either $ \lim_{z \rightarrow z^0} h_{U \cap D} (z) $ or $ \lim_{z \rightarrow z^0} h_{D} (z) $ is bounded away from zero, then 
 \[
 \lim_{z \rightarrow z^0} \frac{h_{U \cap D} (z)}{h_{ D} (z)} = 1. 
 \]
\end{enumerate}

\end{thm}

\medskip

\noindent \textit{Acknowledgements:} The authors would like to thank Kaushal Verma for his advice and constant encouragement. Some of the material presented here has benefitted from conversations the second author had with G.P. Balakumar and Diganta Borah - especially the subsection \ref{SS1}. Special thanks go to S. Baskar and Debanjana Mitra for helpful comments pertaining the theory of ODEs.

\section{Preliminaries}

\subsection{Some definitions}

Let $ \Omega $ be a domain in $ \mathbb{C}^n $, and consider the space $ A^2(\Omega ) $ of all holomorphic functions on $ \Omega $ that are square integrable with respect to the Lebesgue measure on $ \mathbb{C}^n $. If $ A^2(\Omega) \neq \{0\} $, then it is a non-trivial separable Hilbert space, equipped with the $ L^2$-inner product. If $ \{ \phi^j: j \geq 1 \} $ is an orthonormal basis of the Hilbert space $ A^2(\Omega) $, then the Bergman kernel of $ \Omega $ is the function $ K_{\Omega} : \Omega \times \Omega \rightarrow \mathbb{C} $ given by
\begin{equation*}
 K_{\Omega}(z,w) = \sum_{j \geq 1} \phi_j(z) \overline{ \phi_j(w)},
\end{equation*}
where the series on the right hand side converges uniformly on compact sets of $ \Omega \times \Omega $. It is well known that $ K_{\Omega} $ does not depend on the choice of an orthonormal basis for $ A^2(\Omega) $.

If $ K_{\Omega} (z,z) > 0 $ for all $ z \in \Omega$, then 
\[
  \sum_{\mu, \nu=1}^n g^{\Omega}_{\mu \nu}(z)dz_{\mu} d \overline{z}_{\nu}, 
 \]
defines a smooth $ (1,1) $-Hermitian form that is positive semi-definite, where
\begin{equation} \label{E30}
  g_{\mu \nu}^{\Omega}(z) := \frac{\partial^{2}}{\partial z_{\mu}\partial\bar{z_{\nu}}}\log K_{\Omega}(z,z).
 \end{equation} 
 In this case, the associated infinitesimal Bergman (pseudo-)metric at a point $ z \in \Omega $ and a holomorphic tangent vector $ \xi $ at $ z $ is defined as
\[
 b_{\Omega}(z, \xi) = \left( \sum_{\mu, \nu=1}^n g^{\Omega}_{\mu \nu}(z) \xi_{\mu} \overline{\xi}_{\nu} \right)^{1/2}.
\] 
It turns out that $ b_{\Omega} $ is
positive definite if for every $ \xi \in \mathbb{C}^n \setminus \{0\} $, there exists $ f \in A^2(\Omega) $ with the property that that $ df(z) \xi \neq 0 $. S. Bergman showed that the above condition is satisfied for all bounded domains in $ \mathbb{C}^n $, i.e., $ b_{\Omega} $ is a metric for bounded domains in $ \mathbb{C}^n $.  

The Bergman (pseudo-)distance between $ z, w \in \Omega $ is defined as 
\begin{equation*}
 d_{\Omega}^b (z,w)= \inf \int_0^1 b_{\Omega} \big( \gamma(t), \dot{\gamma}(t) \big)\; dt,
\end{equation*}
where the infimum is taken over all piecewise $ C^1$-curves in $ \Omega $ joining $ z$ and $w $.
 

\subsection{Some notation:}

\begin{itemize}
 \item $ \lambda $ denotes the standard Lebesgue measure on $ \mathbb{C}^n $. 
 
 \item $ B^n(z,r) $ is the Euclidean ball centered at point $ z \in \mathbb{C}^n $ and radius $ r > 0 $.
 
 \item For a open set $ \Omega \subset \mathbb{C}^n $ and $ z \in \Omega $, 
 \[ d( z, \Omega) := \inf \left\{ |w-z| : w \in \mathbb{C}^n \setminus \Omega \right\},
 \]
where $ |\cdot| $ denotes the standard Euclidean norm on $ \mathbb{C}^n $. 
 
 \item $ \Delta^2 $ is the unit bidisc $ \{ z \in \mathbb{C}^2: |z_1| < 1, |z_2|<1 \} $.
\end{itemize}


\subsection{Some remarks on the Bergman-Fridman invariant:} In this subsection, we gather some basic properties of the Bergman-Fridman invariant function. 

\begin{lem} \label{L1}
  Let $D\subseteq\clcn$ be a taut domain and $X\subseteq\clcn$ be another  domain. Suppose that there exist two relatively compact sets $K_{1}\subset D$ and $K_{2}\subset X$ and mappings $F^{k} :D\to X$ satisfying the following conditions:
  \begin{enumerate}
   \item [(i)] $F^{k} : D\to F^{k}(D)$ are biholomorphisms for each $k \geq 1$,
   \item [(ii)] for any $k\geq 1$, there exists a point $z^{k}\in K_{1}$ such that $F^{k}(z^{k})\in K_{2}$, and
   \item [(iii)] for any compact set $L\subset X$, there exists a number $s= s(L)$ such that $F^{s}(D)\supset L$.
  \end{enumerate}
 If $ K_X $ is non-vanishing along the diagonal, then $X$ is biholomorphic to $D$.
  \end{lem}
  
   \begin{proof}
   Define $\phi^{k}=(F^{k})^{-1}:F^{k}(D)\to D$  and set $ x^{k} = F^{k}(z^{k})$. Let $\{U_{n}\} $ be an increasing sequence of open sets that exhausts $X$, i.e., $U_{n}\subseteq U_{n + 1}$ for all $n \geq 1$ and $\bigcup U_{n} = X$. Using tautness of $D$, by passing to a subsequence, if necessary, we may assume that, for each $n \geq 1$, there exists an integer $N\in\mathbb{N}$ such that $\{\phi^{k}\}_{k\geq N}$ is defined on $U_{n}$ and converges uniformly on compact subsets to $\phi:U_{n}\to\bar{D}$. It is straightforward to check that $\phi$ is defined on all of $X$ and $\phi(X)\subset\bar{D}$. Further, by passing to a subsequence, (if necessary) we may assume that $\{x^{k}\}$ converges to a point of $x^{0}\in \overline{K}_2 $ and $\{z^{k}\}$ converges to $z^{0}\in \overline{K}_1$. Hence, $\phi(x^{0}) = z^{0}$ and the tautness of $ D $ forces that $ \phi (X) \subset D $.
   
   Next, observe that the $ K_D $ is non-vanishing along the diagonal. Indeed, the transformation rule for the Bergman kernel under the biholomorphisms $ F^{k}: D \rightarrow F^{k}(D) $ gives
   \begin{equation} \label{E1}
    K_D(z,z) = K_{F^k(D)} \left( F^k(z), F^k(z) \right) \left| \det \left(F^k \right)'(z) \right|^2,
   \end{equation}
for each $ k $ and for all $ z $ in $ D $. Here, $ \big( F^k \big)'(z) $ is the complex Jacobian matrix of $ F^k $ at the point $ z $. Further, since $ F^k(D) \subset X $, it follows that
\begin{equation} \label{E2}
 K_X \big( F^k(z), F^k(z) \big) \leq K_{F^k(D)} \big( F^k(z), F^k(z) \big)
 \end{equation}
for each $ k $. Combining (\ref{E1}) and (\ref{E2}) yields
\begin{equation*}
  K_D(z,z) \geq K_X \big( F^k(z), F^k(z) \big) \left| \det \left(F^k \right)'(z) \right|^2 > 0,
\end{equation*}
where the last inequality follows from the facts that $ K_X $ is given to be non-vanishing along the diagonal and each $ F^k $ is a biholomorphism. 
   
Again, since $\phi^{k}:F^{k}(D)\to D$ are biholomorphisms, it follows that
   \begin{equation} \label{E3}
   K_{F^k(D)} ( x^k, x^k ) = K_D (z^k,z^k) \left| \det \left( \phi^k\right)'(x^k) \right|^2,
   \end{equation}
for each $k$. As before, the inclusion $ F^k(D) \subset X $ gives
\begin{equation} \label{E4}
 K_X(x^k,x^k) \leq K_{F^k(D)} (x^k, x^k)
 \end{equation}
for each $ k $. Combining (\ref{E3}), (\ref{E4}) and letting $ k \rightarrow \infty $ gives
\begin{equation*}
 K_X(x^0, x^0) \leq K_D( z^0, z^0) \left| \det  \phi'(x^0) \right|^2.
\end{equation*}
As a consequence, 
\begin{equation} \label{E35}
\left| \det  \phi'(x^0) \right|^2 \geq \frac{K_X(x^0, x^0)}{K_D( z^0, z^0)} > 0.
\end{equation}
Applying the Hurwitz’s theorem to the sequence $ \{ \big(\phi^k\big)' \} $, we deduce that either the Jacobian determinant of $ \phi $ is never zero at any
point of $ X $ or it is identically zero on $ X $. In view of the observation (\ref{E35}), the latter case is ruled out. It follows that
\begin{equation*}
 \left| \det  \phi'(x) \right| > 0
\end{equation*}
for all $ x $ in $ X $ and $ \phi: X \rightarrow D $ is locally one-to-one. 

To prove the injectivity of $\phi$, suppose there are $ x^{1}, x^2 $ in $X$ such that $  \phi( x^{1}) = \phi(x^2) = z$. Consider disjoint neighbourhoods $ V_1 $, $ V_2 $ of $ x^1, x^2 $ respectively in $ X $. Since $\phi$ is locally one-to-one, it follows that $x^{1}$, $x^{2}$ are isolated points of $\phi^{-1}\{z\}$. In this setting, Proposition 5 from Chapter 5 of  \cite{RN} guarantees that  
\[
z \in \phi^k(V_1) \cap  \phi^k(V_2)
\]
for $k$ large. But this contradicts the fact that each $ \phi^k $ is injective, thereby establishing the injectivity of $\phi$. 

In particular, $ X $ can be identified with its biholomorph $ \phi(X) $, which is a subdomain of a taut domain $ D $. It follows that some subsequence of $\{F^{k}\}$ converges uniformly on compact sets of $ D $ and the resultant limit mapping $F:D \rightarrow \overline{X}$ in light of the assumption (iii), and $F(z^{0}) = x^{0}$. Let $L$ be an open subset of $ X $ such that $ K_{2} $ is relatively compact in $ L$. Then 
\begin{equation} \label{E5}
 K_D(z^k,z^k) \leq K_{\phi^k(L)} (z^k,z^k) 
\end{equation}
for $k$ large, and
\begin{equation} \label{E6}
K_{\phi^k(L)} (z^k,z^k) = K_{L}(x^k,x^k) \left| \det \left(F^k \right)'(z^k) \right|^2
\end{equation}
for all $k $ by the transformation rule for the Bergman kernel. The inequalities (\ref{E5}) and (\ref{E6}) together give 
\begin{equation*}
 K_D(z^k,z^k) \leq K_{L}(x^k,x^k) \left| \det \left(F^k \right)'(z^k) \right|^2
\end{equation*}
for all $k$ large. Letting $ k \rightarrow \infty $ in the above inequality yields
\begin{equation*}
 \frac{K_D(z^0,z^0)}{K_L(x^0,x^0)} \leq | \det F'(z^0)|^2,
 \end{equation*}
which, in turn, implies that
\begin{equation} \label{E7}
 | \det F'(z^0)| > 0, 
\end{equation}
since $ K_X (\cdot, \cdot) \leq K_L (\cdot, \cdot) $ and $ K_D $ and $ K_X $ both do not vanish along the diagonal.

An application of the Hurwitz's Theorem to the sequence $ \{ \big(F^k\big)' \} $ guarantees that the Jacobian determinant of $ F $ is never zero at any
point of $ D $, in view of (\ref{E7}). It follows that $ F(D) $ is open and $ F(D) \subset X $. Finally, note that for the mapping $\phi\circ F: D\to D$ and $z\in D$,  
  \begin{equation*}
  \phi\circ F(z) = \lim_{k\to\infty} \phi_{k}\circ F_{k}(z) = z,\end{equation*} 
  so that $ D \subset \phi(X)$. In other words, $\phi : X \rightarrow D $ is a biholomorphism and this completes the proof of the lemma.
 \end{proof} 
 
Note that if the Bergman pseudodistance $ d^b_X $ on a domain $ X \subset \mathbb{C}^n $ is actually a distance, then the topology generated by $ d^b_X $ coincides with the Euclidean toplogy of $ X $.  Now, an argument similar to the one in Proposition 3.3 of \cite{MV} that uses Lemma \ref{L1} and the above fact shows that 

\begin{thm} \label{T6}
Let $ X \subset \mathbb{C}^n $ be a domain such that $d^b_X $ is a distance.  

\begin{enumerate}
 \item [(i)] If $ h_{X}(x^{0}) = 0$ for some point $x^{0}$ of $ X$, then $X$ is biholomorphically equivalent to $ \mathbb{B}^{n}$ and $ h_X \equiv 0 $.
 
 \item [(ii)] If $ X $ is taut and $ h_X (x^0) > 0 $ for some point $x^{0}$ of $ X$, then there is a biholomorphic imbedding $ F: \mathbb{B}^n \rightarrow X $ with the property that $ B_X \big(x^0 , \frac{1}{ h_X(x_0)} \big) \subset F( \mathbb{B}^n ) $.  
 
 \item [(iii)] The function $ h_X( \cdot) $ is continuous on $ X $. 
\end{enumerate}

\end{thm}

\subsection{Localisation Result: } 

In this subsection, we present a proof of Theorem \ref{n10}. It will be useful to recall the Hahn-Lu comparison theorem (\cite{Hahn1},\cite{Hahn2}, \cite{Hahn3}, \cite{Lu}): if $ D $ is a bounded domain in $ \mathbb{C}^n $, then the Carath\'{e}odory distance $ d_D^c $ and the Bergman distance $ d_D^b $ satisfy the inequality $ d_D^c \leq d_D^b $. 


\medskip

\noindent \textit{Proof of Theorem \ref{n10}:} 
Let $ g $ be a local holomorphic peak function at $ z^0 $ defined in a neighbourhood $ U $ of $ z^0 $. Fix $ \epsilon > 0$. Then the localization property of the Bergman metric (\cite{Kim-Yu}) guarantees the existence of a neighbourhood $ U_{1} \subset U $ of $z_{0} $ such that 
\begin{equation} \label{E9}
(1- \epsilon) b_{ U \cap D}(z, v) \leq b_{D}(z, v) \leq (1 + \epsilon) b_{ U \cap D }(z, v)
\end{equation}
for $ z $ in $ U_1 \cap D $ and $v$ a tangent vector at $ z $.

Then for every $ r > 0 $, there is a neighbourhood $ U_2 \subset U_1 $ of $ z^0 $ such that $ B_{U \cap D} (z, r) \subset U_1 \cap D $ for $ z $ in $ U_2 \cap D $. To establish this claim, firstly note that 
\begin{equation} \label{E8}
 d^c_{\Delta} \big( g(z), g(w) \big) \leq d^c_{U \cap D} (z,w) \leq 
 d^b_{U \cap D} (z,w),
\end{equation}
for all $ z, w $ in $ U \cap D $. Since $ g(z^0) = 1 $ and $ |g| < 1 $ on $ U \cap D \setminus \overline{U_1 \cap D} $, it is immediate that 
\begin{equation*}
 d^c_{\Delta} \big( g(z), g(w) \big) \rightarrow + \infty 
\end{equation*}
as $ z \rightarrow z^0 $ for each $ w \in U \cap D \setminus \overline{U_1 \cap D} $. Combining this observation with (\ref{E8}) gives that
\begin{equation*}
 d^b_{U \cap D} \big(z, U \cap D \setminus \overline{U_1 \cap D} \big) \rightarrow + \infty
\end{equation*}
as $ z \rightarrow z^0 $, thereby yielding the desired claim.

Now for a given $ R > 0 $, let $ U_2 $ be a neighbourhood of $ z^0 $ with the property that $ B_{ U \cap D} (z, R) \subset U_1 \cap D $ for $  z \in  U_2 \cap D $. Choose a point $ w \in D $ in the complement of the closure
of $ B_{U \cap D} (z, R) $ and let $ \gamma : [0, 1] \rightarrow D $ be a differentiable path joining $ z $ and $ w $, i.e. $  \gamma(0) = z $ and $ \gamma (1) = w $. Then there is a $ t' \in (0, 1) $ such that $ \gamma \big([0, t') \big) \subset B_{U \cap D} (z, R) $ and $ \gamma (t' ) \in \partial B_{U \cap D} (z, R) $. Using the localization property (\ref{E9}), it follows that 
\begin{eqnarray*}
  R 
  =
  d^b_{U \cap D} \big( z, \gamma(t') \big) 
  \leq 
  \int_0^{t'} b_{U \cap D} \big( \gamma(t), \dot{\gamma}(t) \big) \; dt 
  \leq 
  \frac{1}{(1 - \epsilon)}\int_0^{t'} b_D \big( \gamma(t), \dot{\gamma}(t) \big) \; dt \\
 \leq 
 \frac{1}{(1 - \epsilon)} \int_0^{1} b_D \big( \gamma(t), \dot{\gamma}(t) \big)\; dt, 
\end{eqnarray*}
which in turn implies that $ (1 - \epsilon) R \leq d_D^b (z,w) $ or equivalently that 
\begin{equation} \label{E11}
 B_D \big( z, (1- \epsilon) R \big) \subset B_{U \cap D} (z, R),
\end{equation}
whenever $ z \in U_2 \cap D $. 

Now, let $ f : \mathbb{B}^n \rightarrow U \cap D $ be a biholomorphic imbedding with $ B_{U \cap D} (z, R) \subset  f(\mathbb{B}^n) $. Following (\ref{E11}), it is straightforward that 
\begin{equation*}
 B_D \big( z, (1- \epsilon) R \big) \subset B_{U \cap D} (z, R) \subset f( \mathbb{B}^n ),
\end{equation*}
so that 
\begin{equation} \label{E12}
h_D (z) \leq 1/ \big((1-\epsilon)R\big) 
\end{equation}
whenever $ z \in U_2 \cap D $.

Next, observe the following: For a given $ R > 0 $, fix neighbourhoods $ U_2 \subset U_1 \subset U $ of $ z^0 $ as above. Then $ B_{ U \cap D }\big(z, R/(1 + \epsilon) \big)\subset B_{D}(z, R) $ whenever $ z \in U_2 \cap D $. To see this, pick $ w \in 
B_{ U \cap D }\big(z, R/(1 + \epsilon) \big) $ and consider a differentiable path $ \sigma : [0,1] \rightarrow U \cap D $ with $ \sigma(0) = z, \sigma(1) = w $, such that 
\begin{equation*}
 \int_0^1 b_{U \cap D}  \big( \sigma(t), \dot{\sigma} (t)\big)\; dt < \frac{R}{1+ \epsilon}.
\end{equation*}
Firstly, note that the trace of $ \sigma $ is contained in $ B_{ U \cap D }\big(z, R/(1 + \epsilon) \big) $. For otherwise, there is a $ t_0 \in (0, 1) $ such that $ \sigma \big([0, t_0 ) \big) \subset B_{ U \cap D }\big(z, R/(1 + \epsilon) \big) $ and $ \sigma (t_0 ) \in \partial B_{U \cap D} \big(z, R/( 1 + \epsilon) \big) $, 
and hence
\begin{equation*}
\frac{R}{1+ \epsilon} = d^b_{U \cap D} \big( z, \sigma(t_0) \big) \leq \int_0^{t_0} b_{U \cap D}  \big( \sigma(t), \dot{\sigma} (t)\big)\; dt \leq \int_0^1 b_{U \cap D}  \big( \sigma(t), \dot{\sigma} (t)\big)\; dt < \frac{R}{1+ \epsilon},
\end{equation*}
which is impossible. 

Hence, for $ w \in B_{ U \cap D }\big(z, R/(1 + \epsilon) \big) $ and $ \sigma : [0,1] \rightarrow U \cap D $ a differentiable path with $ \sigma(0) = z, \sigma(1) = w $ such that 
\begin{equation*}
 \int_0^1 b_{U \cap D}  \big( \sigma(t), \dot{\sigma} (t)\big)\; dt < \frac{R}{1+ \epsilon},
\end{equation*}
the image $ \sigma[0,1] $ is contained in $ B_{ U \cap D }\big(z, R/(1 + \epsilon) \big) $ as observed above. If $ z \in U_2 \cap D $, then $ B_{ U \cap D }\big(z, R/(1 + \epsilon) \big) \subset U_1 \cap D $, and the localisation statement (\ref{E9}) applies:
\begin{equation*}
 d_D^b(z, w) \leq \int_0^1 b_{D}  \big( \sigma(t), \dot{\sigma} (t)\big)\; dt \leq (1 + \epsilon) \int_0^1 b_{U \cap D}  \big( \sigma(t), \dot{\sigma} (t)\big)\; dt < R,
\end{equation*}
which forces that $  d_D^b(z, w)  < R $. In other words, 
$ B_{ U \cap D }\big(z, R/(1 + \epsilon) \big)\subset B_{D}(z, R) $ if $ z \in U_2 \cap D $.

Pick $ R > 0 $ such that there is a biholomorphic imbedding
$ f: \mathbb{B}^n \rightarrow D $ with $ B_D (z, R) \subset f (\mathbb{B}^n ) $. By composing with an appropriate automorphism of $ \mathbb{B}^n $, if needed, it may be assumed that $ f (0) = z $. For $ \epsilon > 0 $, it is immediate that $ B_D (z, R - \epsilon) \subset f \big( \mathbb{B}^n (0,r) \big) $ for some $ r \in (0,1) $. Since $ D $  supports a local holomorphic peak function at the point $ z^0 $,  it follows from Lemma 15.2.2 of \cite{Rudin} that there is a neighbourhood $ V $ of $ z^0 $, $ V $ relatively compact in $ U_2 $, such that $ f \big( \mathbb{B}^n (0, r)\big) \subset U \cap D $ if $ z \in  V \cap D $. Then $ \tilde{f}: \mathbb{B}^n \rightarrow  U \cap D $ defined by setting $ \tilde{f}(\zeta) = f(r \zeta) $ is a biholomorphic imbedding of $ \mathbb{B}^n $ into $ U \cap D $ with the property that 
\begin{equation*}
B_{ U \cap D }\Big(z, \frac{R - \epsilon}{1 + \epsilon} \Big) \subset B_D (z, R - \epsilon) \subset \tilde{f}(\mathbb{B}^n ), 
\end{equation*} 
whenever $ z \in V \cap D $. It follows that 
\begin{equation} \label{E13}
 h_{ U \cap D} (z) \leq \frac{1 + \epsilon}{ R - \epsilon }
\end{equation}
for $ z \in V \cap D $. Finally, the result follows by combining the inequalities (\ref{E12}) and (\ref{E13}). 
\qed


\section{$h$-extendible domains}

In this section, we gather the tools and preparatory results required to prove Theorem \ref{T0} (i). This is done in a general setting of h-extendible domains and can be specialised to the cases (i) and (ii) considered in Theorem \ref{T0}. The class of h-extendible domains includes smooth pseudoconvex finite type domains in $ \mathbb{C}^2 $, convex finite type domains in $ \mathbb{C}^n $ and Levi corank one domains in $ \mathbb{C}^n $.




A pseudoconvex domain $ D \subset \mathbb{C}^n $ is said to be h-extendible near  $p^{0}\in \partial{D}$ with multitype $ (1, m_2, \ldots, m_n ) $ if $ \partial D $ is smooth finite type near $ p^0 $ and the Catlin's multitype $ (1, m_2, \ldots, m_n ) $ of $ \partial D $ at $ p^0 $ coincides with the D’Angelo type at the point $ p^0 $. By the definition of multitype, there are local coordinates $ (z_1, z') = (z_1, z_2, \ldots, z_n ) $ in which $ p^0 = (0,0') $ and the domain $ D $ is defined locally near the origin by $ \rho (z) < 0 $, where
\[
 \rho(z) = \Re z_1 + P( z', \overline{z}') + o(|z_1| + |z_2|^{m_2}+ \cdots + |z_n|^{m_n}),
\]
and $P$ is a $ ( 1/m_2, \ldots, 1/m_n )$- weighted homogeneous plurisubharmonic real-valued polynomial in $ \mathbb{C}^{n-1} $ that does not contain any pluriharmonic terms. It is worthwhile mentioning that the above change of coordinates preserve the non-tangential approach to the origin. The associated domain 
\[
D_{\infty} = \Big\{(z_1,z') : \Re z_1 + P(z', \overline{z}') < 0\Big\},
\]
is said to be the local model for $ D $ at $p^{0}$.
Recall from \cite{JY1} and \cite{DiedrichHerbort-1994} that $ D $ is h-extendible at $ p^0 $ iff there exists a $C^{\infty}$-smooth \textit{bumping} function $ a : \mathbb{C}^{n-1} \setminus \{0\} \rightarrow (0, \infty) $  such that  $a$ is  $ ( 1/m_2, \ldots, 1/m_n )$-weighted homogeneous (i.e. the same weights as the polynomial $ P $), and $ P(\cdot) - \epsilon a(\cdot)$ is strictly plurisubharmonic on $\mathbb{C}^{n-1} \setminus \{0\} $ for $ 0 < \epsilon \leq 1$.

Let $ a $ be a bumping function for the domain $ D $ at the point $ p^0 $ as described above. Fix $ \epsilon \in (0, 1) $. Following \cite{BSY},  there is a neighbourhood  $ U_{\epsilon} $ of the 
origin such that $ U_{\epsilon} \cap D $ is contained in the \textit{bumped model} 
\begin{equation} \label{E15}
 D_{\epsilon}= \Big\{(z_1,z') : \Re z_1 + P(z', \overline{z}') - \epsilon a (z', \overline{z}') < 0\Big\},     
\end{equation}
and, moreover
\begin{equation} \label{E14}
 K_{D_{\epsilon}}(z,z) \rightarrow K_{D_{\infty}}(z,z) 
\end{equation}
uniformly on compact subsets of $ D_{\infty} \times D_{\infty} $ as $ \epsilon \rightarrow 0 $.

\medskip

\noindent\textbf{Scaling the domain $D $ within a nontangential cone:} Let $ D \subset \mathbb{C}^{n}$ be a bounded pseudoconvex domain that is h-extendible near  $p^{0}\in \partial{D}$ with multitype $ (1, m_2, \ldots, m_n ) $. Let $ \rho $, $ U_{\epsilon}$, $ P $ and $ a $ be as described above. Let $ \{p^j\} $ be a sequence of points in $ D $ converging to $ (0,0') \in \partial D $ within a nontangential cone $ \Gamma $ with vertex at the origin.

Let $ \pi^j: \mathbb{C}^n \rightarrow \mathbb{C}^n $ be the anisotropic dilation mappings given by 
\begin{equation} \label{E51}
 \pi^j(z) = \Big( \frac{z_1}{|\rho(p^j)|}, \frac{z_2}{|\rho(p^j)|^{1/m_1}}, \ldots, \frac{ z_n}{|\rho(p^j)|^{1/m_n}} \Big),
\end{equation}
and set $ D^j = \pi^j (D) $. Note that the mappings $ \pi^j $ preserve the non-tangential cone convergence. Also, the sequence of domains $ \{D^j \} $ converges in the Hausdorff sense to $ D_{\infty} $, the local model for $ D $ at $ p^0 $. It is known (see \cite{BSY}) that $ K_{D_{\infty}} $ is non-vanishing along the diagonal. Furthermore, since $ D_{\infty} $ is Carath\'{e}odory hyperbolic (refer Theorem 2.3, \cite{Yu-thesis}), it follows from the generalized Hahn-Lu comparison Theorem (Theorem 5.1 of \cite{AGK}) that the Bergman metric on $ D_{\infty} $ is positive-definite.  

\medskip

Furthermore, it should be noted that as $ p^j $ tends to the origin within the cone $ \Gamma $, the points $ \pi^j(p^j) $ converge to a compact subset of the set $ \{ (z_1, z'): \Re z_1 =-1, z'= 0'\} $. 

The next step is to investigate the stability of $ K_{D^j} $ along the diagonal.

\begin{lem}
For $z\in D_{\infty}$, 
$ K_{D^{j}}(z, z)\to K_{D_{\infty}}(z, z) $  as $j\to\infty$.
\end{lem}
\begin{proof}
Fix $z\in D_{\infty}$, $ \epsilon \in (0,1) $ and a neighbourhood $ U_{\epsilon} $ of the origin so that $ U_{\epsilon} \cap D \subset D_{\epsilon} $, where $D_{\epsilon} $ is as defined in (\ref{E15}). Firstly, it suffices to show that 
\begin{equation} \label{E16}
K_{\pi^j( U_{\epsilon} \cap D )}(z,z) \to K_{D_{\infty}}(z, z)
\end{equation}
as $j\to\infty$. Indeed, the transformation formulae for the Bergman kernels under biholomorphisms implies that
\begin{alignat*}{3}
\frac{K_{\pi^{j}(U_{\epsilon} \cap D)}(z, z)}{K_{D^{j}}(z, z)} &=   \frac{K_{\pi^{j}(U_{\epsilon} \cap D)}(z, z)}{K_{\pi^{j}(D)}(z, z)}
&= \frac{K_{U_{\epsilon} \cap D} \left((\pi^{j})^{-1} z,  (\pi^{j})^{-1} z \right)}{K_{D} \left((\pi^{j})^{-1} z,  (\pi^{j})^{-1} z \right)},
\end{alignat*}
for each $ j $. Since $ (\pi^j)^{-1}(z) \rightarrow (0,0') \in \partial D $ as $ j \rightarrow \infty $, and $ D $ supports a local peak function at the origin (refer Theorem 4.1 of \cite{JY1}), and the fact that the Bergman kernels along the diagonals can be localised near local peak points (see Lemma 3.5.2 of \cite{Hor}), it is immediate that
\begin{equation*}
 \lim_{j\to\infty} \frac{K_{U_{\epsilon} \cap D} \left((\pi^{j})^{-1} z,  (\pi^{j})^{-1} z \right)}{K_{D} \left((\pi^{j})^{-1} z,  (\pi^{j})^{-1} z \right)} = 1,
\end{equation*}
and consequently that 
\begin{equation} \label{E21}
\lim_{j\to\infty}\frac{K_{\pi^{j}(U_{\epsilon} \cap D)}(z, z)}{K_{D^{j}}(z, z)} = 1. 
\end{equation}
In order to prove (\ref{E16}), firstly note that 
\begin{equation} \label{E17}
K_{\pi^{j}( U_{\epsilon} \cap D)}(z,z) \leq K_{\pi^{j}( U_{\epsilon} \cap D )\cap D_{\infty}}(z,z), 
\end{equation}
by virtue of the inclusion $ \pi^{j}( U_{\epsilon} \cap D )\cap D_{\infty} \subset
\pi^{j}( U_{\epsilon} \cap D) $ for each $ j $. Moreover, $ \pi^{j}(U_{\epsilon} \cap D) \cap D_{\infty} \rightarrow D_{\infty} $ as $ j \rightarrow \infty $ in the Hausdorff sense. Further, since $ \pi^{j}(U_{\epsilon} \cap D) \cap D_{\infty} \subset D_{\infty} $ for all $j $, a version of Ramadanov's theorem (see the proof of Theorem 12.1.23 of \cite{JP}) implies that
\begin{equation} \label{E18}
 K_{\pi^{j}(U_{\epsilon} \cap D) \cap D_{\infty}}(\cdot, \cdot) \rightarrow K_{D_{\infty}} (\cdot, \cdot) 
\end{equation}
uniformly (along the diagonal) on compact subsets of $ D_{\infty} \times D_{\infty} $. Combining (\ref{E17}) and (\ref{E18}) yields that
\begin{equation} \label{E19}
  \limsup_{j\to\infty} {K_{\pi^{j}( U_{\epsilon}\cap D)}}(z, z) \leq K_{D_{\infty}}(z, z). 
\end{equation}
It remains to verify that 
\begin{equation} \label{E20}
  \liminf_{j\to\infty} {K_{\pi^{j}( U_{\epsilon}\cap D)}}(z, z) \geq K_{D_{\infty}}(z, z). 
\end{equation}
To see this, recall that $ U_{\epsilon} \cap D \subset D_{\epsilon} $ by construction. It follows that $\pi^{j}( U_{\epsilon} \cap D)\subset \pi^{j}(D_{\epsilon}) $ for each $j $. Now, $ D_{\epsilon} $ is invariant under $ \pi^j $ by definition, and hence
\[
 \pi^{j}( U_{\epsilon} \cap D)\subset \pi^{j}(D_{\epsilon}) = D_{\epsilon},
\]
and consequently that
\begin{equation*}
    K_{D_{\epsilon}}(z,z) \leq K_{\pi^{j}( U_{\epsilon} \cap D)}(z, z)
\end{equation*}
for each $j$. Combining the above observation with the localisation statement (\ref{E21}) gives that
\begin{equation*}
 K_{D_{\epsilon}}(z,z) \leq K_{\pi^{j}( U_{\epsilon} \cap D)}(z, z) \leq (1+
 \epsilon) K_{D^j}(z,z) 
\end{equation*}
for all $ j $ large. It follows that
\begin{equation*}
 K_{D_{\epsilon}}(z,z) \leq (1+
 \epsilon) \liminf_{j \rightarrow \infty} K_{D^j}(z,z), 
\end{equation*}
which, in turn, implies that 
\begin{equation} \label{E22}
 K_{D_{\infty}}(z,z) \leq \liminf_{j \rightarrow \infty} K_{D^j}(z,z), 
\end{equation}
owing to (\ref{E14}). To conclude, observe that 
\begin{equation} \label{E23}
 K_{D^j}(z,z) \leq K_{\pi^j(U_{\epsilon} \cap D )}(z,z)
\end{equation}
for each $ j $ since $ \pi^j(U_{\epsilon} \cap D ) \subset D^j $ by definition and the Bergman kernel function decreases when the domain increases. Finally, note that the inequalities (\ref{E22}) and (\ref{E23}) together yield (\ref{E20}). This completes the proof. 
\end{proof}

At this stage, a stability result for the Bergman distance of the scaled domains $ D^j $ is needed. To this end, note that when $ D $ satisfies any of the hypothesis (i)-(iv) of Theorem \ref{T0}, then the associated limiting domain $ D_{\infty} $ is Bergman complete and its Bergman Kernel $ K_{D_{\infty}} $ does not vanish along the diagonal -- a justification for these statements will be provided later when the four cases listed in Theorem \ref{T0} are dealt one-by-one. Moreover, each $ D^j $, being a biholomorph of $ D $, is Bergman complete. 

\section{Convergence of the Bergman distance on the scaled domains}

\begin{thm}\label{L5}
Let $ \Omega^j $ be a sequence of domains in $ \mathbb{C}^n $ converging to another domain $ \Omega_{\infty} \subset \mathbb{C}^n $  in the local Hausdorff sense. If 
\begin{enumerate}
 \item [(a)] $ K_{\Omega_{\infty}} $ is non-vanishing along the diagonal and $ b_{\Omega_{\infty}} $ is positive,
 \item [(b)] $ \Omega^j, \Omega_{\infty} $ are Bergman complete, and
\item [(c)] $ K_{\Omega^{j}}(z, z)\to K_{\Omega_{\infty}}(z, z) $ for $ z \in \Omega_{\infty} $,
\end{enumerate}
then, for $z^{0}\in \Omega_{\infty}$,
\begin{equation*}
d^b_{\Omega^{j}}(z^{0},\cdot) \to d^b_{\Omega_{\infty}}(z^{0},\cdot), 
\end{equation*}
uniformly on compact subsets of $\Omega_{\infty} $. 
\end{thm}

The proof involves several steps. To begin with, observe the following:

\begin{lem} \label{L2}
$K_{\Omega^{j}}(z,w) \to K_{\Omega_{\infty}}(z,w) $ uniformly on compact subsets of $ \Omega_{\infty}\times \Omega_{\infty}$ together with all the derivatives.
\end{lem}

\begin{proof}
Let $ S $ be any relatively compact subset of $ \Omega_{\infty}$. Since $ \{ \Omega^j\} $
converges to $ \Omega_{\infty} $, it follows that $ S \subset \Omega^j $ for all $j$ large and hence, 
\begin{equation} \label{E24}
K_{ \Omega^{j}}(z, z) \leq K_{S}(z, z) 
\end{equation}
for all $z \in S $ and $j$ large. Since 
\begin{align} \label{E25}
|K_{ \Omega^{j}}(z,w)| &\leq \sqrt{K_{ \Omega^{j}}(z, z)}\sqrt{K_{ \Omega^{j}}(w, w)},
\end{align}
for $z, w \in S $ and $j$ large. Since $ S $ was any arbitrary compact subset of $ \Omega_{\infty} $, it follows from the inequalities (\ref{E24}) and (\ref{E25}) that the sequence $\{K_{\Omega^{j}}\}$ is uniformly bounded on compact subsets of $ \Omega_{\infty}\times \Omega_{\infty}$. Hence, the sequence $ \{K_{\Omega^j} \} $ admits a  subsequence (which will be denoted by the same symbols) that converges uniformly on compact subsets of $ \Omega_{\infty}\times \Omega_{\infty}$. In particular, it follows that the limit function $ f (z,w) $, where $ (z,w) \in \Omega_{\infty} \times \Omega_{\infty} $ is holomorphic in the $ z$-variable and conjugate-holomorphic in the $ w$-variable. Further, note that the uniqueness of limits forces that
\begin{equation} \label{E28}
K_{\Omega_{\infty}} (w,w) = f(w,w) 
\end{equation}
(along the diagonal) for $ w \in \Omega_{\infty} $. In particular, the non-vanishing of $ K_{\Omega_{\infty}} $ along the diagonal yields that
\begin{equation*}
f (w,w) > 0,  w \in \Omega_{\infty}.
\end{equation*}
The next step is to invoke the minimizing property of the Bergman kernel to infer that $ f \equiv K_{\Omega_{\infty}}$. To this end, first observe that for $ S $ as above and $ w \in S $ fixed, 
\begin{align} \label{E26}
    \int_S | f(z, w) |^{2} d \lambda (z) 
    \leq \liminf_{j\to\infty}{\int_{S} |K_{\Omega^{j}}(z,w)|^{2}\; d\lambda(z)}
         \leq \liminf_{j\to\infty}{\int_{\Omega^{j}} |K_{\Omega^{j}}(z,w)|^{2}\; d\lambda(z)}.
    \end{align}
Since $ K_{\Omega^j} ( \cdot, w) $ is the reproducing kernel for $ A^2 (\Omega^j ) $, it is immediate that
\begin{align*}
   {\int_{\Omega^{j}} |K_{\Omega^{j}}(z,w)|^{2}\; d\lambda(z)} = K_{\Omega^{j}}(w, w), 
\end{align*}
for each $ j $. So, it follows from (\ref{E26}) that 
\begin{equation*}
\int_S | f(z, w) |^{2} d \lambda (z) \leq  \liminf_{j\to\infty} K_{\Omega^{j}}(w, w).
\end{equation*}
But 
\begin{equation*}  
K_{\Omega^{j}} (w,w) \rightarrow K_{\Omega_{\infty}} (w,w), 
\end{equation*}
(along the diagonal) and hence, 
\begin{equation*} \label{E27}
 \int_S | f(z, w) |^{2} d \lambda (z) \leq  K_{\Omega_{\infty}} (w,w).
\end{equation*}
Combining the above observation with (\ref{E28}) renders that
\begin{alignat}{3} \label{E29}
\int_{S} \left| \frac{f(z, w)}{f(w,w)} \right|^{2} d \lambda (z) \leq \frac{1}{ K_{\Omega_{\infty}} (w,w)}.
\end{alignat}
Since 
\begin{equation*}
 K_{\Omega_{\infty}} (w,w) = \sup \left\{ \Big( \int_{\Omega_{\infty}} |h(\zeta)|^2 \; d \lambda (\zeta) \Big)^{-1}: h \in A^2(\Omega_{\infty}), h(w)= 1 \right\},
\end{equation*}
it is immediate from (\ref{E29}) that 
\begin{equation*}
\int_{S} \left| \frac{f(z, w)}{f(w,w)} \right|^{2} d \lambda (z) \leq \int_{\Omega_{\infty}} |h(\zeta)|^2 \; d \lambda (\zeta), 
\end{equation*}
and consequently that
\begin{equation*}
\int_{\Omega_{\infty}} \left| \frac{f(z, w)}{f(w,w)} \right|^{2} d \lambda (z) \leq \int_{\Omega_{\infty}} |h(\zeta)|^2 \; d \lambda (\zeta), 
\end{equation*}
for every $ h \in A^2(\Omega_{\infty}) $ with $ h(w)= 1 $. 
This is just the assertion that  $ f( \cdot, w)/ f(w,w) $ is an solution in $ A^2(\Omega_{\infty}) $ of the variational problem
\[
\min \int_{\Omega_{\infty}} |g(z)|^{2}\; d\lambda (z), 
\]
for $ g \in A^{2}(\Omega_{\infty}) $ and $ g(w) =1 $.
But the function $K_{\Omega_{\infty}}(\cdot, w)/K_{\Omega_{\infty}}(w, w)$ is the unique solution in $ A^2(\Omega_{\infty}) $ of this extremal problem, and hence, it follows that $ f(\cdot, w) = K_{\Omega_{\infty}}(\cdot, w) $. Since $ w $ was any arbitrary point of $ \Omega_{\infty}$, it follows that $ f \equiv K_{\Omega_{\infty}} $. The above reasoning also shows that every convergent subsequence of $ \{ K_{\Omega^{j}} \}$ has the same limit $K_{\Omega_{\infty}}$ and hence, the sequence $ \{ K_{\Omega^{j}} \} $ itself converges to $ K_{\Omega_{\infty}}$. Since each of the functions $K_{\Omega^{j}}(z, w)$ are harmonic, the convergence of the corresponding derivatives follows. 
\end{proof}

Observe that Lemma \ref{L2} together with the facts that $ K_{\Omega^{j}}$ and $K_{\Omega_{\infty}}$ are non-vanishing along the diagonal gives the following result on the stability of the Bergman metric tensors (recall the definition from (\ref{E30})). 

\begin{cor} \label{C1}
For $ \mu, \nu = 1, \ldots, n $, 
$ g_{\mu \nu}^{\Omega^j} \rightarrow g_{\mu \nu}^{\Omega_{\infty}} $ as $ j \rightarrow \infty $ uniformly on compact subsets of $\Omega_{\infty}$ together with all the derivatives.
\end{cor}

An immediate consequence of the above corollary is the convergence at the level of the infinitesimal Bergman metric for $ \Omega^j $. More precisely, it follows that

\begin{lem}\label{L3}
 $b_{\Omega^{j}}\to b_{\Omega_{\infty}}$ uniformly on compact subsets of $ \Omega_{\infty}\times\mathbb{C}^{n}$.
  \end{lem}
  
 \begin{proof}
  Let $ S \subset  \Omega_{\infty}$ be compact. Then $ S $ is relatively compact in $ \Omega^{j}$ for $j$ large. For $ \xi $ a holomorphic tangent vector at $ z \in S$, consider
  \begin{align} \nonumber
    b_{\Omega^{j}}(z,\xi) - b_{\Omega_{\infty}}(z,\xi) &= 
    \left( \sum_{\mu, \nu=1}^n g^{\Omega^j}_{\mu \nu}(z) \xi_{\mu} \overline{\xi}_{\nu} \right)^{1/2} - \left( \sum_{\mu, \nu=1}^n g^{\Omega_{\infty}}_{\mu \nu}(z) \xi_{\mu} \overline{\xi}_{\nu} \right)^{1/2}  
    \\
    &= \label{E34}
    \frac{{\xi} \textbf{ G}^{j}(z) \bar{\xi}^{T} - {\xi} \textbf{G}(z) \bar{\xi}^{T}}{({\xi} \textbf{G}^{j}(z) \bar{\xi}^{T})^{1/2} + ({\xi} \textbf{G}(z) \bar{\xi}^{T})^{1/2}},
 \end{align}
where $ \textbf{G}^{j}(\cdot) $ and $ \textbf{G}( \cdot) $ are the $ n \times n $ matrices $ \left( g_{\mu \nu}^{\Omega^{j}}(\cdot) \right) $  and $ \left( g_{\mu \nu}^{\Omega_{\infty}}(\cdot) \right) $ respectively. Moreover, the positivity of $ b_{\Omega_{\infty}} $ renders an $ \lambda_0 > 0 $ such that 
\begin{equation} \label{E31}
| {\xi} \textbf{G}(z) \bar{\xi}^{T}| \geq \lambda_0 |\xi|^{2}, 
\end{equation}
for all $z \in S $. Further,
\begin{equation} \label{E32}
\|\textbf{G}^{j} - \textbf{G} \| \rightarrow 0
\end{equation}
uniformly on $S$ by virtue of Corollary \ref{C1}. Here, 
$\|{\cdot}\|$ denotes the operator norm on the space of $ n \times n $ matrices. It follows from (\ref{E31}) and (\ref{E32}) that  
\begin{align} \label{E33}
    |{\xi} \textbf{G}^{j}(z) \bar{\xi}^{T}| \geq 
    |{\xi} \textbf{G}(z) \bar{\xi}^{T}| - |{\xi} ( \textbf{G}^{j} - \textbf{G})(z) \bar{\xi}^{T}|
    \geq \frac{\lambda_0|\xi|^{2}}{2}
    \end{align}
for all $ j $ large. 

The desired convergence now follows by combining the observations (\ref{E32}), (\ref{E33}) and (\ref{E31}) together with (\ref{E34}).
 \end{proof}

For notational convenience, we write $ g_{\mu \nu}^{\Omega_{\infty}} $ as simply $ g_{\mu \nu} $ in the sequel. Regard $ \Omega^{j}$ and $ \Omega_{\infty}$ as Riemannian domains in $\mathbb{R}^{2n}$. Writing the real co-ordinates for $ z = (x_{1}+ \iota x_{2},\ldots,x_{2n-1}+ \iota x_{2n}) \simeq (x_{1}, x_{2},\ldots, x_{2n-1}, x_{2n})$ and $ \xi = (y_{1} + \iota y_{2},\ldots, y_{2n - 1} + \iota y_{2n}) \simeq (y_{1}, y_{2},\ldots, y_{2n-1}, y_{2n}) $, it is immediate that
 
\[
\sum_{\mu, \nu = 1}^{n} g_{\mu \nu}(z) {\xi}_{\mu} \bar{{\xi}_{\nu}} = \sum_{\mu, \nu = 1}^{2n} \tilde{g_{\mu \nu}}(x) y_{\mu} y_{\nu}
\]
where $ \tilde{g_{\mu\nu}}$ is either $ \Re {g_{\mu' \nu'}}$ or $ \Im g_{\mu' \nu'}$ for $ 1 \leq \mu', \nu' \leq n $. Note that the $ n \times n $ matrix $ \textbf{G} = \left( g_{\mu \nu} \right) $ is positive definite iff the $ 2n \times 2n $ matrix $ \tilde{\textbf{G}} = \left( \tilde{g}_{\mu \nu} \right)$ is positive definite. 

At this stage, note that a stability result for the Bergman distances on $ \Omega^j $ relies on the corresponding statement for the Christoffel symbols of $ \Omega^j $ and $ \Omega_{\infty} $. Recall that the Christoffel symbols for the Riemannian connection in terms of the $ \tilde{g_{\mu \nu}}$ for $ \Omega_{\infty}$ are defined as
 \begin{equation}\label{eqn:A}
 \Gamma^{\eta, \Omega_{\infty}}_{\mu \nu} = 
 \frac{1}{2}\sum_{\tau = 1}^{2n}\Big\{\pdv{\tilde{g}_{\nu \tau}}{x_{\mu}} + \pdv{\tilde{g}_{\tau \mu} }{x_{\nu}} - \pdv{\tilde{g}_{\mu \nu}}{x_{\tau}}\Big\} \tilde{g}^{\tau \eta},
 \end{equation} 
 where $ \tilde{g}^{\tau \eta}$ is the $ (\tau, \eta)$-entry of the matrix $ {\tilde{\textbf{G}}}^{-1}$ and $ \mu, \nu, \eta = 1, \ldots, 2n $. The  Christoffel symbols $ \Gamma^{\eta, \Omega^j}_{\mu \nu} $ for $ \Omega^j $ are defined analogously. The next result is obtained as a consequence of Corollary \ref{C1} and the fact the Bergman metric on  $ \Omega^j$ and $ \Omega_{\infty} $ are positive definite.

\begin{lem}\label{L4}
For $ \mu, \nu, \eta = 1, \ldots, 2n $, $ \Gamma^{\eta, \Omega^j}_{\mu \nu} \rightarrow \Gamma^{\eta, \Omega_{\infty}}_{\mu \nu} $ as $ j \rightarrow \infty $ uniformly on compact subsets of $ \Omega_{\infty}$.
\end{lem}

We are now in a position to provide:

\medskip

\noindent \textit{Proof of Theorem \ref{L5}:}
 Suppose that the assertion of the lemma is false. Then there exists an $ \epsilon_0 >0 $,  a compact set $ S \subset \Omega_{\infty}$, and  points $ q^{j} \in S $ such that
 \begin{equation*}
 |d^b_{\Omega^{j}}(z^{0}, q^{j}) - d^b_{\Omega_{\infty}}(z^{0}, q^{j}) | > \epsilon_0 
 \end{equation*}
 for all $j$ large. Note that $ S $ is compactly contained in $ \Omega^j $ for all $ j $ large, and consequently that, the points $ q^j \in \Omega^j $ for all $ j $ large.
 After passing to a subsequence, if needed, assume that $ q^{j } \rightarrow q^{0} \in S $. Further, the continuity of $d^b_{ \Omega_{\infty}}(z^{0},\cdot) $ guarantees that 
 \begin{equation*}
  d^b_{\Omega_{\infty}}(z^{0}, q^{j}) \rightarrow d^b_{\Omega_{\infty}}(z^{0}, q^0)
 \end{equation*}
and hence,  
 \begin{equation}\label{eqn:E}
    |d^b_{\Omega^{j}}(z^{0}, q^{j}) - d^b_{\Omega_{\infty}}(z^{0}, q^{0}) | > \epsilon_0/2. 
 \end{equation}
 For $ \epsilon > 0 $ fixed, choose a piecewise $ C^1$-smooth path  $\gamma : [0,1]\to \Omega_{\infty}$ such that $\gamma(0) = z^{0}$, $\gamma(1) = q^{0}$ and 
 \begin{equation*}
 \int^{1}_{0} b_{\Omega_{\infty}}(\gamma(t), \dot{\gamma}(t)) dt < d^b_{\Omega_{\infty}}(z^{0}, q^{0}) + \epsilon/2.
 \end{equation*}
 Consider $\gamma^{j} : [0,1]\to \mathbb{C}^n $ defined by setting 
 \[
 \gamma^{j}(t) = \gamma(t) + (q^{j} - q^{0}) t. 
 \]
It follows that $\gamma^{j}: [0,1] \rightarrow \Omega^j $ for $j$ large, $\gamma^{j}(0) = z^{0}$ and $\gamma^{j}(1) = q^{j}$ and both $\gamma^{j} \rightarrow \gamma$ and $\dot{\gamma}^{j}\rightarrow \dot{\gamma}$ uniformly on $[0,1]$. In this setting, Lemma \ref{L3} implies that 
 \begin{align*}
     d^b_{\Omega^{j}}(z^{0}, q^{j})  
     \leq \int^{1}_{0} b_{\Omega^{j}}(\gamma^{j}(t), \dot{\gamma}^{j}(t)) dt
     \leq  \int^{1}_{0} b_{\Omega_{\infty}}(\gamma(t), \dot{\gamma}(t)) dt + \epsilon/2 
     < d^b_{\Omega_{\infty}}(z^{0}, q^{0}) + \epsilon,
 \end{align*}
 for all $ j $ large. Therefore,
 \begin{equation} \label{E36}
 \limsup_{j \rightarrow \infty} {d^b_{\Omega^{j}}(z^{0}, q^{j})} \leq d^b_{\Omega_{\infty}}(z^{0}, q^{0}).
 \end{equation}
 It remains to verify that
 \begin{equation}\label{E37}
 d^b_{\Omega_{\infty}}(z^{0},q^{0}) \leq \liminf_{j \rightarrow \infty} {d_{\Omega^{j}}(z^{0}, q^{j})},
\end{equation}
which together with (\ref{E36}) will give
\[
\lim_{j\to\infty} d^b_{\Omega^{j}}(z^{0}, q^{j}) = d^b_{\Omega_{\infty}}(z^{0}, q^{0}),
\]
thereby contradicting $(\ref{eqn:E})$ and proving the desired result.

To establish (\ref{E37}), first recall that each $ \Omega^{j} $ is Bergman complete by assumption. Hence, there are geodesics $ \sigma^{j}: [0,1] \rightarrow \Omega^{j}$ joining $z^0 $ and $q^j$, i.e., $ \sigma^{j}(0) = z^{0}$, $ \sigma^{j}(1) = q^{j} $ and 
\begin{equation} \label{E40}
\int^{1}_{0} b_{\Omega^{j}}(\sigma^{j}(t), \dot{\sigma}^{j}(t)) dt = d^b_{\Omega^{j}}(z^{0}, q^{j}). 
\end{equation}
 Setting $X^{j} = \dot{\sigma}^{j}(0)$,  we first show that $\{X^{j}\}$ is a bounded subset of $\mathbb{R}^{2n}$. Suppose not, then there is a subsequence (which we will denote by the same symbols) such that $ |{X^{j}}|\to \infty$, where $ |{\cdot}| $ denotes the Euclidean norm on $ \mathbb{R}^{2n} $. Consider another norm 
norm $\norm{\cdot}_{z^{0}, \Omega^{j}}$ can be given on $ \mathbb{R}^{2n} $ as
follows
\[
\norm{X}_{z^{0}, \Omega^{j}} = b_{\Omega^{j}}(z^{0}, X) = \big(\sum_{\mu, \nu}\tilde{g}^{\Omega^j}_{\mu \nu}(z^{0}) X_{\mu}X_{\nu}\big)^{1/2} 
\]
for $ X \in \mathbb{R}^{2n}$ and $ j \in \mathbb{N} $. Note that 
 \begin{align*}
\norm{X^{j}}^{2}_{z^{0}, \Omega^{j}} 
= X^{j} \mathbf{G}^{j} (X^{j})^{T} 
= X^{j} \mathbf{G} (X^{j})^{T} + X^{j} (\mathbf{G}^{j} - \mathbf{G}) (X^{j})^{T} 
 \end{align*}
 where $ \textbf{G}^{j} $ and $ \textbf{G} $ are the $ 2n \times 2n $ matrices $ \left( \tilde{g}_{\mu \nu}^{\Omega^{j}}(z^0) \right) $  and $ \left( \tilde{g}_{\mu \nu}^{\Omega_{\infty}}(z^0) \right) $ respectively. Next, the positivity of $ b_{\Omega_{\infty}} $ guarantees an $ \lambda_0 > 0 $ such that 
 \begin{align*}
\norm{X^{j}}^{2}_{z^{0}, \Omega^{j}} 
\geq \lambda_0 |{X^{j}}|^{2} - \norm{ \mathbf{G}^{j} - \mathbf{G}} |{X^{j}}|^{2}. 
 \end{align*}
The above observation together with Corollary \ref{C1} implies that 
 \[
\norm{X^{j}}^{2}_{z^{0}, \Omega^{j}} \geq  \lambda_0/2 | {X^{j}} | ^{2}
 \]
for all $ j$ large. In particular, it follows that $\norm{X^{j}}_{z^{0},\Omega^{j}}\to \infty$ as well. Writing 
\[
X^{j} = s^{j} V^{j},
\]
where $ \norm{V^{j}}_{z^{0},\Omega^{j}} = 1$ and $ s^{j} = \norm{X^{j}}_{z^{0},\Omega^{j}}\to\infty$. Denote by $ C_{{X}^{j}}$, the maximal geodesic in $ \Omega^{j}$ starting from $ z^{0}$ with $\dot{C}_{{X}^{j}}(0) = X^{j}$. Since each $ \Omega^j $ is  Bergman complete, it follows that $ C_{{X}^{j}}(t) $ is defined for all values of the parameter $ t \in \mathbb{R} $. In particular, 
\begin{equation} \label{E38}
\sigma^{j}(t) = C_{X^{j}}(t),
\end{equation}
for all $ t \in [0,1] $. Moreover, since $ C_{X^j} $ is a geodesic, it is immediate that 
$ b_{\Omega^j} \left( C_{X^j}(t), \dot{C}_{{X}^{j}} (t)\right)  $ is constant for all $ t \in \mathbb{R} $. As a consequence,
\begin{equation*}
 b_{\Omega^j} \left( C_{X^j}(t), \dot{C}_{{X}^{j}} (t)\right) = b_{\Omega^j} \left( C_{X^j}(0), \dot{C}_{{X}^{j}} (0)\right) = b_{\Omega^j} \left( z^0, X^j \right) 
 \end{equation*}
for $ t \in \mathbb{R} $ and $j \in \mathbb{N} $. Further, since the Bergman metric $ b_{\Omega^j} ( z^0, \cdot ) $ is homogeneous, 
it follows that
\begin{equation*}
 b_{\Omega^j} \left( z^0, X^j \right) = b_{\Omega^j} \left( z^0, s^j V^j \right)
 = s^j b_{\Omega^j} \left( z^0, V^j \right) = s^j,
\end{equation*}
as $ b_{\Omega^j} \left( z^0, V^j \right) = 1 $ by construction. Hence,
\begin{equation} \label{E41}
 b_{\Omega^j} \left( C_{X^j}(t), \dot{C}_{{X}^{j}} (t)\right) = s^j
\end{equation}
for $ t \in \mathbb{R} $ and $j \in \mathbb{N} $. 
Now, combining equations (\ref{E40}), (\ref{E38}) and (\ref{E41}) yields that
 \begin{align*}
 d^b_{\Omega^{j}}(z^{0}, p^{j}) = \int^{1}_{0} b_{\Omega^{j}}(\sigma^{j}(t), \dot{\sigma}^{j}(t)) dt 
 = \int^{1}_{0} b_{\Omega^{j}} \left( C_{X^{j}}(t), \dot{C}_{X^{j}}(t) \right) dt
 = \int^{1}_{0} s^j dt = s^j,
 \end{align*}
 which, in conjunction with (\ref{E36}), implies that
 \begin{equation*}
 s^{j} = d^b_{\Omega^{j}}(z^{0}, q^{j}) \leq d^b_{\Omega_{\infty}}(z^{0}, q^0 ) + \epsilon 
 \end{equation*}
 for $j$ large. This is a contradiction  since $ \{ s^{j} \}$ is unbounded. Hence, $\{X^{j}\}$ is a bounded subset of $ \mathbb{R}^{2n}$. 
 
 By choosing a subsequence, which we again denote by the same symbols, we may assume that $ X^{j} \to X $ for some $ X \in \mathbb{R}^{2n}$. Let $t \longmapsto \left( \sigma_0(t) , \dot{\sigma_0} (t) \right)$ be the unique solution of the first order system
\begin{eqnarray} \label{E42}
\begin{cases}
 \displaystyle\frac{d x_{\eta}}{ dt} = y_{\eta},  \; & \left( x_1(0), \ldots, x_{2n}(0) \right)  = z^0, 
 \\
  \displaystyle\frac{d y_{\eta}}{ dt} = - \sum_{\mu,\nu} \Gamma_{\mu \nu}^{\eta, \Omega_{\infty}} y_{\mu} y_{\nu}, \; & \left( y_1(0), \ldots, y_{2n}(0) \right)  = X,
 \end{cases}
 \end{eqnarray}
on an interval around the origin, for $ \eta = 1, \ldots, 2n $. Note that the right hand side of the differential equation (\ref{E42}) satisfies a Lipschitz condition with respect to $ (x_1, \ldots, x_{2n}, y_1, \ldots, y_{2n}) $ on a set containing $ {\sigma_0}([0,1]) \times \dot{\sigma_0}([0,1])$ because of the continuity of the derivatives of the Christoffel symbols, thereby rendering the uniqueness of the solution of (\ref{E42}). 
Further, $ \sigma_0 $ satisfies the second order system
\begin{equation*}
 \frac{d^2 x_{\eta}}{ dt^2} + \sum_{\mu,\nu} \Gamma_{\mu \nu}^{\eta, \Omega_{\infty}} \frac{d x_{\mu}}{dt}\frac{d x_{\nu}}{dt} = 0,
 \end{equation*}
for $ \eta = 1, \ldots, 2n $, with $ \sigma_0(0) = z^{0}$ and $ \dot{\sigma}_{0}(0) = X $, i.e., $\sigma_0 $ is a geodesic in $ \Omega_{\infty} $ starting at $ z^{0}$ with initial velocity $X$. Moreover, since $ ( \Omega_{\infty}, d^b_{\Omega_{\infty}} ) $ is 
complete, the geodesic $ \sigma_0 $ starting from $ z^0 $ is defined for all values of the parameter $ t \in \mathbb{R} $. 

Note that the geodesics $ \sigma^j $, as defined by (\ref{E40}), satisfy the first order system
\begin{eqnarray} \label{E43}
\begin{cases}
 \displaystyle\frac{d x_{\eta}}{ dt} = y_{\eta}, \; & \left( x_1(0), \ldots, x_{2n}(0) \right)  = z^0, \\  
 \displaystyle\frac{d y_{\eta}}{ dt} = - \sum_{\mu,\nu} \Gamma_{\mu \nu}^{\eta, \Omega^j} y_{\mu} y_{\nu},  \; & \left( y_1(0), \ldots, y_{2n}(0) \right)  = X^j.
 \end{cases}
 \end{eqnarray}
for $ \eta = 1, \ldots, 2n $. 

Next we invoke the continuous dependence of solution of an initial value problem on the right hand side of the differential equation and on the initial data to deduce that $\sigma^{j}\to \sigma_0 $ and $\dot{\sigma}^{j}\to\dot{\sigma_0}$ uniformly on $[0,1]$. At this stage, recall that $ X^j \rightarrow X $ by construction and $ \Gamma^{\eta, \Omega^j}_{\mu \nu} \rightarrow \Gamma^{\eta, \Omega_{\infty}}_{\mu \nu} $, for $ \eta = 1, \ldots, 2n $, as $ j \rightarrow \infty $ courtesy Lemma \ref{L4}. Now, applying the Theorem on continuous dependence (see, for instance, Theorem VI, Section 12, Chapter III of \cite{Walter}), it follows that every solution of the \textit{perturbed} initial value problem (\ref{E43}) \textit{stays} near the unique solution of the initial value problem (\ref{E42}) on $ [0,1] $. In particular,
\begin{equation*}
 \sigma^j \rightarrow \sigma_0 \mbox{ and } \dot{\sigma}^j \rightarrow \dot{\sigma}_0
\end{equation*} 
on $ [0,1] $ as $ j \rightarrow \infty $. Moreover, the convergence is uniform across the entire interval $ [0, 1] $ follows from the estimates (see, for instance, Theorem V, Section 12, Chapter III of \cite{Walter}) known from the theory of ordinary differential equations. 

Finally, the above observation together with Lemma \ref{L3} ensures that 
\begin{equation} \label{E44}
 b_{\Omega^j}(\sigma^j(t),\dot{\sigma}^j(t)) \rightarrow b_{\Omega_{\infty}}(\sigma_0(t),\dot{\sigma}_0(t))
\end{equation}
holds uniformly in $t $ on $[0,1]$.  On the other hand, since
\[
\sigma_0(1) = \lim_{j\to\infty} \sigma^{j}(1) = \lim_{j\to\infty} q^{j} = q^{0}, 
\]
 it follows that $ \sigma_0 $ is a $C^1$-smooth curve in $ \Omega_{\infty} $ joining $ z^{0}$ and $ q^{0}$. Hence,
 \[
 d^b_{\Omega_{\infty}}(z^{0},q^{0}) \leq \int^{1}_{0} b_{\Omega_{\infty}}(\sigma_0(t),\dot{\sigma}_0(t)) dt \leq \int^{1}_{0} b_{\Omega^{j}}(\sigma^{j}(t),\dot{\sigma}^{j}(t)) dt + \epsilon
= d^b_{\Omega^{j}}(z^{0}, q^{j}) + \ep, 
 \]
 for all $ j $ large. The second inequality above follows from (\ref{E44}) and the last one is immediate from (\ref{E40}). As a consequence, 
\[
   d^b_{\Omega_{\infty}}(z^{0},q^{0}) \leq \liminf_{j \rightarrow \infty} {d^b_{\Omega^{j}}(z^{0}, q^{j})}, 
\]
as desired. This completes the proof.
\qed

\begin{cor}\label{L6} Let $ \Omega^j, \Omega_{\infty} $ be domains in $ \mathbb{C}^n $ satisfying the hypothesis of Theorem \ref{L5}. Let $ z^0 \in \Omega_{\infty} $ and $ R > 0 $ be fixed. Then, for any sequence of points $ z^j \rightarrow z^0 $, the Bergman balls $ B_{\Omega^{j}}(z^{j},R )$ converge to $B_{\Omega_{\infty}}(z^{0},R) $ in the local Hausdorff sense. Moreover, for $ \epsilon > 0$, the inclusions $ B_{\Omega_{\infty}}(z^{0},R)\subset B_{\Omega^{j}}(z^{j}, R + \ep)$ and $ B_{\Omega^{j}}(z^{j},R - \ep)\subset B_{\Omega_{\infty}}(z^{0},R)$ hold for $j$ large.
\end{cor} 
 
\begin{proof} Let $ z^j $ be a sequence of points converging to $ z^0 \in \Omega_{\infty} $. Then 
\begin{equation} \label{E45}
d^b_{\Omega^{j}}(z^{j},\cdot) \to d^b_{\Omega_{\infty}}(z^{0},\cdot)
\end{equation}
uniformly on compact subsets of $ \Omega_{\infty}$. Indeed, for all $ w $ in a fixed compact subset $ S $ of $ \Omega_{\infty} $,
\begin{equation*}
 | d^b_{\Omega^{j}}(z^{j},w) - d^b_{\Omega_{\infty}}(z^{0}, w) | \leq  | d^b_{\Omega^{j}}(z^{j},w) - d^b_{\Omega^j}(z^{0}, w) | + | d^b_{\Omega^{j}}(z^0,w) - d^b_{\Omega_{\infty}}(z^{0}, w) |,
\end{equation*}
where the first summand
\begin{eqnarray*}
 | d^b_{\Omega^{j}}(z^{j},w) - d^b_{\Omega^j}(z^{0}, w) | \leq d^b_{\Omega^j} (z^j, z^0) \rightarrow d_{\Omega_{\infty}}(z^0, z^0) =0, 
\end{eqnarray*}
and the second summand $ | d^b_{\Omega^{j}}(z^0,w) - d^b_{\Omega_{\infty}}(z^{0}, w) | \rightarrow 0 $ for all $j $ large courtesy Lemma \ref{L5}.

Next, to establish the convergence
\begin{equation} \label{E46}
B_{\Omega^{j}}(z^{j},R ) \rightarrow B_{\Omega_{\infty}}(z^{0},R), 
\end{equation}
consider a compact subset $ L $ of $ B_{\Omega_{\infty}}(z^{0},R) $. It follows that $ L $ is relatively compactly contained in $ \Omega^j $ for all $ j $ large and
\begin{equation*}
 d^b_{\Omega_{\infty}} (z^0, w) < C, 
\end{equation*}
for  $ w \in L $ and for some $ C = C(L) \in (0, R) $. Further, it follows from (\ref{E45}) that for $ C^* \in (C, R) $, 
\begin{equation*}
d^b_{\Omega^{j}}(z^j,w) < d^b_{\Omega_{\infty}}(z^0,w) + C^* - C,
\end{equation*}
and, consequently that 
\begin{equation*}
d^b_{\Omega^{j}}(z^j,w) < C^* < R,
\end{equation*}
for $ w \in L $ and $ j $ large . This exactly means that $ L $ is compactly contained in $ B_{\Omega^{j}}(z^{j},R )  $ for all $ j $ large. Conversely, let $ L $ be a compact subset of $ \mathbb{C}^n$ such that $L$ is contained in the interior of $\bigcap_{j\geq j_0} B_{\Omega^{j}}(z^j, R)$ for some $j_0$. It is immediate that $L $ is compactly contained in $ \Omega_{\infty}$ and there exists $ {C}^{**} = {C}^{**}(L) \in (0, R) $ with the property that
\begin{equation*}
d^b_{ \Omega^{j}}(z^{j}, w)  < {C}^{**}
\end{equation*}
for $w \in L$ and $j$ large. Again invoking (\ref{E45}), it follows that
for $ C^{***} \in ( C^{**}, R) $, 
\[
d^b_{ \Omega_{\infty}}(z^{0}, w) < d^b_{\Omega^{j}}(z^{j},w) + {C}^{***} - C^{**}                                                                         
\]
for all $w \in L $ and $j$ large. As a consequence,
\[
d^b_{\Omega_{\infty}}(z^{0}, w) < C^{***} < R 
\]
for $ w \in L $, or equivalently that $ L $ is compactly contained in $B_{ \Omega_{\infty}}(z^{0}, R)$. This assertion verifies (\ref{E46}).

Next, note that $B_{\Omega_{\infty}}(z^{0}, R)$ is relatively compact in $ \Omega_{\infty}$ since $ \Omega_{\infty}$ is Bergman complete. So, for $ \epsilon > 0 $ fixed, (\ref{E45}) yields that
\begin{equation*}
d^b_{\Omega^{j}}(z^{j},w) < d^b_{\Omega_{\infty}}(z^{0},w) + \epsilon 
\end{equation*}
holds for all $w \in B_{\Omega_{\infty}}(z^{0},R)$ and $j$ large. This is just the assertion
\[
B_{\Omega_{\infty}}(z^{0},R) \subset B_{\Omega^{j}}(z^{j},R + \ep),
\]
for $ j $ large. Finally, suppose that the balls $ B_{\Omega^{j}}(z^{j}, R - \ep) $ are not contained in $ B_{\Omega_{\infty}}(z^{0},R)$  for $ j $ large.  
Then there is an $\ep_{0} > 0$ and points $a^{j}\in B_{\Omega^{j}}(z^{j}, R - \ep_{0})$ such that $ a^{j} $ are on the boundary of $ B_{\Omega_{\infty}}(z^{0}, R)$. Using the compactness of $\partial B_{\Omega_{\infty}}(z^{0}, R)$, we may assume that $a^{j}\to a^0 \in \partial B_{\Omega_{\infty}}(z^{0}, R)$. In this setting, it follows from lemma (\ref{E45}) that 
\[
d^b_{\Omega^{j}}(z^{j}, a^j)\to d^b_{\Omega_{\infty}}(z^0, a^{0}), 
\]
which, in turn, implies that $ d^b_{\Omega_{\infty}}(z^0, a^{0})\leq R - \ep_{0}$. This contradicts the fact that $ a^0 \in \partial B_{\Omega_{\infty}}(z^{0}, R)$, thereby completing the proof.
 
\end{proof}

\section{Proof of Theorem \ref{T0}(i) - Levi corank one domain $ D $} 

The core of our proof comprises the verification that $ D $ as in Theorem \ref{T0}(i) satisfies the conditions stated in Theorem \ref{L5}. 

\subsection{Completeness of $ (D_{\infty}, d^b_{D_{\infty}}) $:} \label{SS1}

When $ \Omega \subset \mathbb{C}^n $ is a Levi corank one domain and $ p^j $ is any \textit{arbitrary} sequence converging to $ p^0 \in \partial \Omega $, refer \cite{Cho1}, \cite{Cho2}, \cite{TT} to see that the scaling technique applies. I.e. there are biholomorphic mappings $A^j: \Omega \to \Omega^j$ such that $ \Omega^j$ converge in the local Hausdorff sense to a polynomial domain 
\begin{equation} \label{E50}
 \Omega_{\infty}=\Big\{ z \in \mathbb{C}^n : 2 \Re z_1 + P_{\infty} \left(z_n, \overline{z}_n \right) + \sum_{j=2}^{n-1} \vert z_j \vert^2 < 0 \Big\},
\end{equation}
where $ P_{\infty} \left(z_n, \overline{z}_n \right) $ is a real-valued subharmonic polynomial of degree at most $ 2m $ without any harmonic terms, $ 2m $ being the 1-type of $ \partial \Omega $ at $ p^0 $ ($m\geq 1$ is a positive integer). In order to assert that $ \Omega_{\infty} $ is Bergman complete, we first need to show that $ K_{\Omega_{\infty}} $ is non-vanishing along the diagonal. 

To prove that $ K_{\Omega_{\infty}} (z, z) \neq 0 $, firstly note that $ P_{\infty} $, being a subharmonic polynomial, is of even degree, say $ 2k $. Now, let $ P_{2k} $ denote the homogeneous part of $ P_{\infty} $ of degree $ 2k $ and set $ L = P_{\infty} - P_{2k} $. Then $ P_{2k} $ is a homogeneous subharmonic polynomial of degree $ 2k $ without any harmonic terms. In this setting, \cite{BF} guarantees the existence of an $ \epsilon > 0 $ and a nowhere vanishing holomorphic peak function 
\begin{equation*}
f \in \mathcal{O}({\Omega^*}), \; \Omega^*= \left\{ (w_1,w_2) \in \mathbb{C}^2: 2 \Re w_1 + P_{2k} \left(w_2, \overline{w}_2 \right) < \epsilon \left(|w_1| + |w_2|^{2k}\right) \right\}
\end{equation*}
at the origin for $ \mathcal{O}({\Omega^*}) $, that is continuous on $ \overline{\Omega^*} $ with exponential decay at infinity. Among other things, it was also shown in  \cite{BF} that
\begin{equation} \label{E48}
 \exp \left( -C_1 \big(|w_1| + |w_2|^{2k}\big)^{1/N} \right)  \leq |f(w_1, w_2)| \leq \exp \left( -C_2 \big(|w_1| + |w_2|^{2k} \big)^{1/N} \right)
\end{equation}
holds for some positive constants $ C_1, C_2 $, some integer $ N $ and for all $ (w_1, w_2) \in \overline{\Omega^*} $. 

Further, note that as the degree of the polynomial $ L $ is stricly less than $ 2k $, it is immediate that $ L (z_n, \overline{z}_n) < \epsilon |z_n|^{2k} $ for $ |z_n| $ large, say $ |z_n| > R_{\epsilon} $. On the other hand, let $ 2 c $ be an upper bound for $ L (z_n, \overline{z}_n) $ on the compact disc $ |z_n| \leq R_{\epsilon} $. Therefore,
\begin{equation} \label{E47}
| L(z_n, \overline{z}_n) |  < \epsilon |z_n|^{2k} + 2c, 
\end{equation}
holds for all $ z_n \in \mathbb{C} $. 

To emulate this construction for $ \Omega_{\infty} $, consider the projection $ \Pi: \mathbb{C}^n \rightarrow \mathbb{C}^2 $ defined by $ \Pi (z_1, z_2, \ldots, z_n) = (z_1, z_n) $. Observe that for the positive constant $ c $ as chosen in (\ref{E47}),  
the translation $ T_c: (z_1, z_n) \longmapsto (z_1-c, z_n) $ maps $ \Pi(\Omega_{\infty})= \{ (z_1, z_n) \in \mathbb{C}^2 : 2 \Re z_1 + P_{\infty} (z_n, \overline{z}_n) < 0 \} $ into $ \Omega^* $. Indeed, every $ ( z_1,z_n ) \in \Pi(\Omega_{\infty}) $ satisfies
\begin{equation*}
 2 \Re z_1 + P_{\infty} (z_n, \overline{z}_n) < 0,
\end{equation*}
or equivalently
\begin{equation*}
2 \Re z_1 + P_{2k} (z_n, \overline{z}_n) + L (z_n, \overline{z}_n) < 0,  
\end{equation*}
which implies that
\begin{equation*}
2 \Re z_1 + P_{2k} (z_n, \overline{z}_n) < - L (z_n, \overline{z}_n) < \epsilon |z_n|^{2k} + 2c,
\end{equation*}
using (\ref{E47}). The above inequality can be rewritten as
\begin{equation*}
2 \Re (z_1 - c) + P_{2k} (z_n, \overline{z}_n) 
< \epsilon |z_n|^{2k} <  \epsilon \big(|z_1 - c| + |z_n|^{2k}\big),
\end{equation*}
which is just the assertion that $ T_c \circ \Pi(\Omega_{\infty}) \subset \Omega^* $. Hence, $ \tilde{f} = f \circ T_c \circ \Pi $ is a nowhere vanishing holomorphic function on $ \Omega_{\infty} $. Moreover, it follows from the polynomial inequality defining $ \Omega_{\infty} $ (refer (\ref{E50})) that $ \max\{ |z_1, |z_n| \} \rightarrow \infty $ on $ \Omega_{\infty} $ as $ z $ tends to   infinity from within $ \Omega_{\infty} $. As a consequence of the above observation and (\ref{E48}), $ \tilde{f} $ also has exponential decay at infinity. Furthermore, $ \tilde{f} $ is square integrable on $ \Omega_{\infty} $ since the domain $ \Omega_{\infty} $ is defined by a polynomial inequality. Therefore, the Bergman kernel $ K_{\Omega_{\infty}} $ is nowhere vanishing along the diagonal. 

Before going further, it is worthwhile noting that 
\begin{equation} \label{E49}
 \Omega_{\infty} \subset \Omega^{**} = \Big\{ z \in \mathbb{C}^n : 2 \Re (z_1 - c) + P_{2k} (z_n, \overline{z}_n)  + \sum_{j=2}^{n-1} \vert z_j \vert^2
<  \epsilon \big(|z_1 - c| + |z_n|^{2k}\big) \Big\},
\end{equation}
and that $ \Omega^{**} $ admits a (global) holomorphic peak function $ f^*$ at $ (c,0') \in \partial \Omega^{**} $ due to \cite{BF}. Moreover, $ f^{*} $ is nowhere zero and decays exponentially fast at infinity.

It remains to show that $ \Omega_{\infty} $ is Bergman complete. 
Since the Carath\'{e}odory metric is smaller than the Bergman metric and $ K_{\Omega_{\infty}} $ is non-vanishing along the diagonal, it suffices to show (see Theorem 5.1 of \cite{AGK}) that $ \Omega_{\infty} $ is complete with respect to the inner Carath\'{e}odory distance. Recall that $ \Omega_{\infty} $ equipped with the inner Carath\'{e}odory distance is complete if it supports (global) holomorphic peak functions at every boundary point. The following technical theorem from \cite{AGK} will be useful to obtain the peak functions for the unbounded domain $ \Omega_{\infty} $.

\begin{thm} \label{T3}
Let $ \Omega $ be a domain in $ \mathbb{C}^n $. If $ p \in \partial \Omega $ satisfies the following two properties:
\begin{enumerate}
 \item [(i)] There exists an open neighbourhood $ V $ of $ p $ in $ \mathbb{C}^n $ and a holomorphic function $ g$ defined in an open neighbourhood of the closure of $ V \cap \Omega $ that is also a peak function at $ p $ for $ \mathcal{O}(V) $.
 
 \item [(ii)] There are constants $ r_1, r_2, r_3 $ with $ 0 < r_1 < r_2 < r_3 < 1 $ and $ B^n(p,r_3) \subset V $, and there exists a Stein neighbourhood $ U $ of $ \overline{ \Omega } $ and a function $ h \in \mathcal{O}( \Omega \cup V), h \neq 0 $ on $ V $ satisfying
 \begin{equation*}
\{ z \in V : g(z) = 1 \} \cap U \cap \overline{ B^n(p,r_2)} \setminus B^n(p,r_1) = \emptyset,
 \end{equation*}
and
\begin{equation*}
 |h(z)|^2 \leq C_0 \frac{\big( \min\{1, d(z, U) \}\big)^{2n} }{\big( 1 + \|z\|^2 \big)^2}, \; \mbox{for all} \; z \in \Omega,
\end{equation*}
for some positive constant $ C_0 $, 
\end{enumerate}
then $ \Omega $ admits a holomorphic peak function at $ p $.
\end{thm}

Apply the above theorem to the domain $ \Omega_{\infty} $. Recall that each finite boundary point of $ \partial \Omega_{\infty} $ is a local holomorphic peak point. Let $ \Omega^{**} $ (as defined by (\ref{E49})) play the role of the Stein neighbourhood $ U $ and $ f^* $ play the role of the function $ h $.
Then Theorem \ref{T3} provides a global holomorphic peak function at each finite boundary point of $ \partial \Omega_{\infty} $. Since the point at infinity in $ \partial \Omega_{\infty} $ is a global holomorphic peak point for $ \Omega_{\infty} $ (refer Lemma 1 of \cite{BP}), it follows that $ (\Omega_{\infty}, d^b_{\Omega_{\infty}}) $ is complete.

In particular, if $ D $ is $ C^{\infty} $-smooth Levi corank one domain in $ \mathbb{C}^n $ as in Theorem \ref{T0} (i), then the above observation shows that the limiting domain $ D_{\infty} $ (obtained by scaling $ D $ along a sequence $ p^j \rightarrow p^0 \in \partial D $ from within a nontangential cone $ \Gamma $ having vertex at $ p^0 $) is Bergman complete. 

\medskip

\textit{Proof of Theorem \ref{T0}(i):} To understand the behaviour of $ h_{D}(p^{j}) $ as $ j \rightarrow \infty $, the following two scenarios need to be examined (after passing to a subsequence, if necessary): 

\begin{enumerate}
 \item [(a)] $ h_{D}(p^{j}) \rightarrow 0$, or
 
 \item [(b)] $ h_{D}(p^{j}) $ is bounded below by a positive constant $ C_0$.
\end{enumerate}

In Case (a), the Bergman completeness of $ D_{\infty} $, Corollary \ref{L6} together with arguments similar to those employed in Lemma 3.3 of \cite{BMV} show that $ D_{\infty} $ is biholomorphic to $ \mathbb{B}^n $. We present the outline of a proof here for the sake of completeness.

For $ \epsilon > 0 $ fixed, there is $ R_j > 0 $ such that
\[
 \frac{1}{R_j} < h_D(p^j) + \epsilon
\]
for each $j $. Since $ h_{D}(p^{j}) \rightarrow 0$, by hypothesis, it is immediate that $ R_j \rightarrow \infty $. Moreover, there are 
biholomorphic imbeddings $ F^{j}: \mathbb{B}^{n} \to D $ with the property that $ F^j(0) = p^j $ and $ B_D (p^{j}, R) \subset F^{j}(\mathbb{B}^{n}) $. Next, consider the mapppings 
\[
 \pi^{j} \circ F^{j}: \mathbb{B}^{n} \to D^j, 
\]
where $ \pi^j $ are automorphisms of $ \mathbb{C}^n $ associated with the sequence $ \{ p^j \} $ by the scaling method as defined in (\ref{E51}). Then $ \pi^{j}  \circ F^{j} (0) = \pi^j(p^{j}) = q^j \rightarrow q^0 = (-1, 0') \in D_{\infty} $. Applying Theorem 3.11 of \cite{TT}, it follows that $ \{ \pi^{j}  \circ F^{j} \} $ is a normal family. Hence, some subsequence of $ \{ \pi^{j}  \circ F^{j} \} $ (which will be denoted by the same symbols) converges uniformly on compact sets of $ \mathbb{B}^n $ to a holomorphic mapping $ \Phi: \mathbb{B}^n \rightarrow D_{\infty} $. It turns out that $ \Phi $ is a biholomorphism from $ \mathbb{B}^n $ onto $ D_{\infty} $. This can be seen by considering 
\[
 (\pi^{j} \circ F^{j})^{-1}: D^j \rightarrow \mathbb{B}^{n}, 
\]
which admits a subsequence converging uniformly on compact sets on compact subsets of $ D_{\infty} $ to holomorphic mapping $ \Psi: D_{\infty} \rightarrow \mathbb{B}^n $. It can be seen that $ \Phi $ and $ \Psi $ are inverses of each other, so that $ D_{\infty} $ is biholomorphic on $ \mathbb{B}^n $. In particular, $ h_{D_{\infty}}(\cdot) \equiv 0 $.

\medskip

In Case (b), the limiting domain $ D_{\infty} $ will not be biholomorphic to $ \mathbb{B}^n $ unlike case (a). Here, the arguments are 
similar to the ones in Theorem 1.2(ii) of \cite{BMV}. However, for clarity and completeness, we provide a proof here.

Since $ h_D(p^j) = h_{D^j}(q^j) $, we study the behaviour of $ h_{D^j}(q^j) $ as $ j \rightarrow \infty $. The goal is to show that $ h_{D^j}(q^j) 
\rightarrow h_{D_{\infty}} ( q^0 ) $. 

For $ \epsilon > 0 $ fixed, let $ R > 0 $ be such that
 $ 1/{R} < h_{D_{\infty}} (q^0) + \epsilon $
and $ F : \mathbb{B}^n \rightarrow D_{\infty} $ be a biholomorphic imbedding with the property that $ F(0) = q^0 $ and $ B_{D_{\infty}} (q^0, R) \subset F(\mathbb{B}^n ) $. Let $ \delta > 0 $ be such that $ B_{D_{\infty}} (q^0, R - \epsilon) \subset F \big({B}^n (0, 1 -\delta) \big) \subset D^j $ for all $ j $ large. It then follows from Corollary \ref{L6} that
\begin{equation*}
B_{D^j} (q^j, R - 2 \epsilon)  \subset B_{D_{\infty}} (q^0, R - \epsilon) \subset F \big({B}^n (0, 1 -\delta) \big) \subset D^j,
\end{equation*}
which, in turn, implies that
\begin{equation*}
h_{D^j}(q^j) \leq \frac{1}{R - 2 \epsilon} 
\end{equation*}
for all $ j $ large. Hence,
\begin{equation} \label{E54}
\limsup_{j\rightarrow \infty} h_{D^j}(q^j) \leq h_{D_{\infty}} (q^0). 
\end{equation}
Next, consider biholomorphic imbeddings $ F^j : \mathbb{B}^n \rightarrow D^j $ and $ R_j > 0 $ such that $ F^j (0) = q^j $, $ B_{D^j}(q^j, R_j) \subset F^j(\mathbb{B}^n) $ and
$ 1/{R_j} < h_{D^j} (q^j) + \epsilon $. 
We claim that $ F^j $ admits a convergent subsequence. 
To establish this, consider $ (\pi^j)^{-1} \circ F^j : \mathbb{B}^n \rightarrow D $ where $ \pi^j $ are the scalings associated with the sequence $ \{ p^j \} $ as in (\ref{E51}). Then $ (\pi^j)^{-1} \circ F^j (0) = p^j \rightarrow p^0 \in \partial D $. Now, owing to Theorem 3.11 of \cite{TT}, $ F^j $ is a normal family. It follows that the uniform limit $ F: \mathbb{B}^n \rightarrow D_{\infty} $ and $ F(0) = q^0 $. Further, it follows from (\ref{E54}) that 
\begin{equation*} 
 \frac{1}{R_j} < h_{D^j} (q^j) + \epsilon < h_{D_{\infty}} (q^0) + 2 \epsilon, 
\end{equation*}
for all $ j $ large. Furthermore, $ h_{D^j} (q^j) $ is bounded below by a positive constant $ C_0$ by assumption, and therefore, some subequence of $ R_j $ 
(which will be denoted by the same symbols) converges and its limit $ R_0 $ is a positive real number. It then follows from Corollary \ref{L6} that
\begin{equation*}
 B_{D_{\infty}} (q^0, R_0 - 2 \epsilon) \subset B_{D^j}(q^j, R_0 - \epsilon) \subset B_{D^j}(q^j, R_j) \subset F^j(\mathbb{B}^n), 
\end{equation*}
for all $ j $ large. This implies that $ B_{D_{\infty}} (q^0, R_0 - 2 \epsilon) \subset F^j(\mathbb{B}^n) $. Also, Hurwitz's theorem ensures that $ F $ is one-one. 

To summarize, $ F: \mathbb{B}^n \rightarrow D_{\infty} $ is a bihlolomorphic imbedding such that 
$ F (\mathbb{B}^n) \supset B_{D_{\infty}} (q^0, R_0 - 2 \epsilon) $. Hence, $ h_{D_{\infty}} (q^0) \leq 1/(R_0 - 2 \epsilon) $, or equivalently that
\begin{equation*} 
\liminf_{j\rightarrow \infty} h_{D^j}(q^j) \geq h_{D_{\infty}} (q^0). 
\end{equation*}
This completes the proof of Theorem \ref{T0} (i).

\qed


\section{Proof of Theorem \ref{T0}(ii) - Strongly pseudoconvex domain $ D $}
   
If $ D $ is $ C^2 $-smooth strongly pseudoconvex domain in $ \mathbb{C}^n $ as in Theorem \ref{T0} (i), then the scaling technique introduced by S. Pinchuk (see \cite{Pinchuk}) provides biholomorphisms $ A^j: D \to D^j$ such that $ D^j $ converge in the local Hausdorff sense to an unbounded domain $ D_{\infty} \subset \mathbb{C}^n$ and $ q^j:= A^j(p^j) = (-1,0') \in D_{\infty}$. It is worthwhile mentioning the associated limiting domain $ D_{\infty} $ is uniquely determined. In fact,  
   \[
    D_{\infty} = \{ z \in \mathbb{C}^n : 2 \Re z_1 + \sum_{j=2}^n |z_j|^2 < 0 \} 
   \]
is biholomorphic to $ \mathbb{B}^n $. As a consequence, $ K_{D_{\infty}} $ is non-vanishing along the diagonal and $ (D_{\infty}, d^b_{D_{\infty}}) $ is complete and $ b_{D_{\infty}} $ is positive. 

Here, Lemma 2.1 of \cite{BBMV} provides the stability result for the Bergman Kernels $ K_{D^j} (\cdot, \cdot) $ analogous to Lemma \ref{L2} in the current setting.
%
Hence, Theorem \ref{L4} and Corollary \ref{L6} applies
to yield the stability of the Bergman distances $ d^b_{D^j} $ and corresponding Bergman balls. Finally, use the facts that $ h_D(p^j) = h_{D^j}\big((-1,0') \big) $ and that the Bergman ball $ B_{D_{\infty}} \big( (-1,0'), R \big) $ for every $ R > 0 $ is biholomorphic to $ \mathbb{B}^n $ to conclude that $ h_D(p^j ) \rightarrow 0 $. The proof of this last statement proceeds exactly as that of Theorem 4.1 of \cite{MV} and is therefore omitted.

\section{Proof of Theorem \ref{T0}(iii) - Convex domain $ D $}

If $ D $ is a $ C^{\infty} $-smooth convex domain as in Theorem \ref{T0} (iii), then there are two cases to be considered:

\begin{enumerate}
 \item [(a)] $ \partial D $ is of finite type near $ p^0 $, or
 
 \item [(b)] $ \partial D $ is of infinite type near $ p^0 $.
\end{enumerate}

In Case (a), according to J. McNeal (\cite{McNeal}), T. Hefer (\cite{Hefer}) (also, refer \cite{NPT}), there exist biholomorphisms $ A^j: D \to D^j$, where $ D_j $ are convex domains that converge in the local Hausdorff sense to 
\[
D_0 = \Big \{ z \in \mathbb{C}^n : 
\Re \Big ( \displaystyle\sum_{k=1}^n b_k z_k  \Big) + P(z', \overline{z}') < 1
\Big \},
\]
$ b_k $ are complex numbers and $ P$ is a real convex
polynomial of degree less than or equal to $2m$ ($ 2m $ being the $1$-type of $ \partial D $ at $ p^0 $). Also, $ q^j:= A^j(p^j) = (0,0') \in D_0 $. It is known that $ D_0 $ is biholomorphically equivalent to a bounded domain contained in the unit polydisc in $ \mathbb{C}^n $. As a consequence, $ K_{D_0} $ is non-vanishing along the diagonal. Furthermore, $ b_1 \neq 0 $, and hence $ D_0 $ is biholomorphic
to 
\[ 
D_{\infty} = \big \{ z \in
\mathbb{C}^n : 2 \Re z_1 + P(z', \overline{z}')  < 0 \big \}, 
\]
via a mapping that sends $ (0, 0') \in D_0 $ to $ q^0 = (-1,0') \in D_{\infty} $. Further, $ D_{\infty} $ is convex and Kobayashi complete hyperbolic. It follows by Theorem 2.6 of \cite{Gaussier-Zimmer} that the Kobayashi and the Carath\'{e}odory metrics coincide on $ D_{\infty} $. In particular, $ D_{\infty} $ equipped with the inner Carath\'{e}odory metric is complete hyperbolic. It follows from Theorem 5.1 of \cite{AGK} that $ D_{\infty} $ is Bergman complete and $ b_{D_{\infty}} $ is positive definite.

\medskip

\noindent Case (b): it follows from \cite{Zim2} that there are biholomorphisms $ A^{j} $ defined on $ D $ such that $D^{j} = A^{j}(D)$ converge in the local Hausdorff sense to a convex domain $D_{\infty}$ and the points $ q^j:= A^j(p^j) $ converge to $ q^0 = (0,0') \in D_{\infty} $. Further, 
$ D_{\infty} $ does not contain any non-trivial complex affine lines (refer the proof of Proposition 6.1 of \cite{Zim1} and hence, $ D_{\infty} $ is Kobayashi complete hyperbolic. From this point, the completeness of $ (D_{\infty}, d^b_{D_{\infty}} ) $ and positivity of $ b_{D_{\infty}}$ follows exactly as in case (a).

The convergence of the Bergman kernels $ K_{D^j} $ is guaranteed by Theorem 10.1 of \cite{Gaussier-Zimmer} in both the cases (a) and (b). 
Hence, Theorem \ref{L5} and consequently, Corollary \ref{L6} is applicable here. It then follows using Proposition 4.2 of \cite{Zim1} that  $ \lim_{j \rightarrow \infty} h_D(p^j) = h_{D_{\infty}} (q^0)  $ in both the cases (a) and (b). The reasonings are similar as those in employed the prooving Theorem \ref{T0}(i)  and we shall not repeat the argument here.
 
 \section{Proof of Theoerem \ref{T0} (iv) - Strongly polyhedral domain $ D $:}\label{poly}  

\begin{defn} \label{D1}

A bounded domain $ \Omega \subset \mathbb{C}^n $ is said to be a strongly pseudoconvex polyhedral domain with piecewise smooth boundary if 
there exist $ l (\geq 2) $ real valued  $ C^2$-smooth functions 
$ \rho_1, \ldots, \rho_l : \mathbb{C}^n \rightarrow \mathbb{R} $ such that
\begin{itemize}
 \item [(i)] $ \Omega = \big\{ z \in \mathbb{C}^n: \rho_1(z) < 0, \ldots, \rho_l(z) < 0 \big\} $,
 \item [(ii)] for $ \{ i_1, \ldots, i_k \} \subset \{1, \ldots, l\} $, the gradient vectors $ \nabla \rho_{i_1}(p), \ldots, \nabla \rho_{i_k}(p) $ are linearly independent 
 over $ \mathbb{C} $ for every point $ p $ satisfying $ \rho_{i_1}(p) = \ldots = 
 \rho_{i_k}(p) = 0 $, and
 \item [(iii)] $ \partial \Omega $ is strongly pseudoconvex at every smooth boundary point.
\end{itemize} 
 
\end{defn}

Such a domain $ \Omega $ is necessarily pseudoconvex, since 
the intersection of finitely many domains of holomorphy is a domain of holomorphy. Further, there exist peak functions at each boundary point, and hence, any polyhedron domain as in Definition (\ref{D1}) is Carath\'{e}odory complete and consequently, Bergman complete.

If $ D \subset \mathbb{C}^2 $ is a polyhedral domain as in Definition \ref{D1}, then the proof of Theorem \ref{T0} (iv) divides into two parts:

\begin{enumerate}
 \item [(a)] $ \partial D $ is a smooth near $ p^0 $, or
 \item [(b)] $ \partial D $ is a non-smooth near $ p^0 $.
\end{enumerate}

In Case (a), $ \lim_{z \rightarrow p^0} h_D(z) = 0 $ using the proof of Theorem \ref{T0} (ii) and the localisation statement -- Theorem \ref{n10}.

In Case (b), applying the scaling method from \cite{Kim-Yu} -- there are biholomorphisms  $A^j: D \to D^j $ from $ D $ onto the scaled domains $ D^j$ such that $ D^j$ converge in the local Hausdorff sense to a domain $ D_{\infty} \subset \mathbb{C}^2 $ which is one of $ \mathbb{B}^2 $, the bidisc $ {\Delta}^2 $, or a Siegel domain of second kind (\cite{Piatetski-Shapiro}) given by 
\begin{equation} \label{Sie}
\tilde{D}= \left\{ z \in \mathbb{C}^2: \Re z_1 + 1 > \frac{ Q_1(z_2)}{m^2}, \Re z_2 > -1 \right\}, 
\end{equation}
where $ m > 0 $ and $ Q_1 $ is a strictly subharmonic polynomial of degree $ 2 $, and $q^j:=A^j(p^j) \to q^0 \in D_{\infty}$. 

It follows from \cite{Piatetski-Shapiro} that $ \tilde{D} $ is 
biholomorphic to a bounded domain in $ \mathbb{C}^2 $. In particular, the Bergman kernel $ K_{\tilde{D}} $ is non-vanishing along the diagonal. Moreover, $ \tilde{D} $ can be seen as the intersection of an open ball with a half space in $ \mathbb{C}^2 $ and hence, $ \tilde{D} $ is an unbounded convex domain. Furthermore, $ \tilde{D} $ is Kobayashi complete hyperbolic. As before, Theorem 2.6 of \cite{Gaussier-Zimmer} ensures that the Kobayashi and the Carath\'{e}odory metrics coincide on $ \tilde{D} $. As a consequence, $ \tilde{D} $ is complete with respect to the inner Carath\'{e}odory distance, which, in turn, implies that $ \tilde{D} $ is Bergman complete (see Theorem 5.1 of \cite{AGK}).

To summarize, in Case (b), the limit domain $ D_{\infty} $ is one of $ \mathbb{B}^2 $, $ {\Delta}^2 $, or a Siegel domain (as described by (\ref{Sie})). In particular, in each of these three cases, $ K_{D_{\infty}} $ is non-vanishing along the diagonal and $ (D_{\infty}, d^b_{D_\infty}) $ is Bergman complete. Since the limiting domain $ D_{\infty} $ is convex, Lemma 2.1 of \cite{BBMV} provides an analogue of Lemma \ref{L2} in the current setting.

Finally, Theorem \ref{L5} and Corollary \ref{L6} and the fact that $ D^j \subset 2 D_{\infty} $ for all $ j $ large ensure that $ \lim_{j \rightarrow \infty} h_D(p^j) = h_{D_{\infty}} (q^0) \big) $ in case (b) as before. This completes the proof of Theorem \ref{T0} (iv).
 
\section{Detecting Strong pseudoconvexity}

The proofs of Theorems \ref{T4} and \ref{T5} proceed much like that of Theorem 1.1 of \cite{MV2}. We include a concise proof here for clarity and completeness. 

\medskip

\noindent \textit{Proof of Theorems \ref{T4} and \ref{T5}:}
Let $ p^j \rightarrow p^0 \in \partial D $ along the inward normal to $ \partial D $ at $ p^0 $. Apply the scalings $ A^j : D \rightarrow D^j $ associated to the domain $ D $ and the sequence $ p^j \rightarrow p^0 $, 
so that the rescaled domains $ D^j $ converge to a limiting domain $ D_{\infty} $ and $ A^j(p^j) \rightarrow q^0 \in D_{\infty} $.  More specifically, if $ D $ is a Levi corank one domain, then the associated limit domain 
\begin{equation} \label{E52}
 D_{\infty}=\Big\{ z \in \mathbb{C}^n : 2 \Re z_1 + Q \left(z_n, \overline{z}_n \right) + \sum_{j=2}^{n-1} \vert z_j \vert^2 < 0 \Big\}.
\end{equation}
Moreover, since $ p^j \rightarrow p^0 $ normally, it follows that $ Q $ is a homogeneous subharmonic polynomial of degree $ 2m $ ($ 2m $ being the $1$-type of $ \partial D $ at $ p^0 $, where $m\geq 1$ is a positive integer). Furthermore, $ Q $ tallies with the polynomial of same degree in the homogeneous Taylor expansion of the defining function for $ \partial D $ around $ p^0 $.  Similarily, if $ D $ is of convex finite type near $ p^0 $, then 
\begin{equation} \label{E53}
D_{\infty} = \big \{ z \in
\mathbb{C}^n : 2 \Re z_1 + P(z', \overline{z}')  < 0 \big \}, 
\end{equation}
where $ P $ is a real convex polynomial. As before, the convergence $ p^j \rightarrow p^0 $ along the inner normal forces that $ P $ is of degree $ 2m' $, where $2m' $ is the $1$-type of $ \partial D $ at $ p^0 $,  and $ P $ is precisely the polynomial of degree $ 2m'$ that appears in the homogeneous Taylor expansion of the defining function for $ \partial D $ around $ p^0 $. 

Now, $ \lim_{j \rightarrow \infty} h_{D} (p^j) = 0 $ by hypothesis, and hence, it is immediate from Theorem \ref{T0} that $ h_{D_{\infty}} (q^0) = 0 $. But then Theorem \ref{T6} enforces that $ D_{\infty} $ must be biholomorphically equivalent to $ \mathbb{B}^n $. 

If $ D $ is convex infinite type near $ p^0 $, then $ (D_{\infty}, d^b_{D_{\infty}}) $ fails to be Gromov hyperbolic (cf. Theorem 3.1 and Proposition 6.1 of \cite{Zim1}). This, in turn, implies that there is no isometry of Bergman metrics between $ D_{\infty} $ and $ \mathbb{B}^n $.
This is a contradiction since $ D_{\infty} $ is biholomorphic to $ \mathbb{B}^n $. Hence, $ \partial D $ cannot be of infinite type near $ p^0 $, which implies that  $ \partial D $ has to be finite type near $ p^0 $.

If $ D $ is either Levi corank domain or of convex finite type near $ p^0 $, since the corresponding limit domain $ D_{\infty} $ (as described in (\ref{E52}) and (\ref{E53}) respectively) is a biholomorph of $ \mathbb{B}^n $, it follows that $ D_{\infty} $  is biholomorphic to an \textit{half-plane in $\mathbb{C}^n $}, namely to the unbounded representation of $ \mathbb{B}^n $

\[
 \Sigma = \big\{ z \in \mathbb{C}^n : 2 \Re z_1 + \abs{z_2}^2 + \abs{z_3}^2 + \ldots + 
\abs{z_{n}}^2 < 0 \big\}
\]
via an appropriate Cayley transform. Let $ \theta : D_{\infty} \rightarrow \Sigma $ be a biholomorphism from $ D_{\infty} $ onto $ \Sigma $. Moreover, it may be assumed that cluster set of $ {\theta} $ at some point
$ (\iota \alpha, 0') \in \partial D_{\infty} $ (where $ \alpha \in \mathbb{R} $) contains a  
point of $ \partial \Sigma $ different from the point at infinity on $ \partial \Sigma $. In this setting, Theorem 2.1 of \cite{Coupet-Pinchuk} applies so that $ {\theta} $ extends biholomorphically past the boundary of $ D_{\infty} $ to a neighbourhood  of the point $ (\iota \alpha, 0') \in \partial D_{\infty} $. Next, since both $ D_{\infty} $ and $ \Sigma $ are invariant under the translations of the form $ z \mapsto z + \iota s $, $ s \in \mathbb{R} $, it may be assumed that $ (\iota \alpha, 0') $ is the origin and $ \theta \big( (0,0') \big) = \big( (0,0') \big) \in \partial \Sigma $. Since the Levi form is preserved under biholomorphisms around a boundary point, it follows that $ \partial D_{\infty} $ is strongly pseudoconvex near the origin. This is just the assertion that $ \partial D $ is strongly pseudoconvex near $ p^0 $. Hence the result.
\qed

\end{document}